\documentclass[12pt,a4paper]{amsart}
\usepackage[mathscr]{eucal}
\usepackage{mathrsfs}
\usepackage{amsmath, amsthm}
\usepackage{amsfonts, amssymb}
\usepackage{graphicx}
\usepackage{color}
\usepackage[normalem]{ulem}

\newcommand{\Real}{\mathbb{R}}
\newcommand{\Natural}{\mathbb{N}}

\newcommand{\D}{\mathrm{d}}
\newcommand{\vol}{\mathrm{vol}}
\newcommand{\Ric}{\mathrm{Ric}}
\newcommand{\Hess}{\mathrm{Hess}}
\newcommand{\para}{/\!\!/}

\newcommand{\M}{m^{\prime}}
\newcommand{\MM}{m^{\prime\prime}}
\newcommand{\T}{t^{\prime}}
\newcommand{\TT}{t^{\prime\prime}}
\newcommand{\B}{\tau^{\prime}}
\newcommand{\BB}{\tau^{\prime\prime}}

\theoremstyle{definition}
\newtheorem{Def}{Definition}[section]

\newtheorem{Rm}{Remark}[section]

\theoremstyle{plain}
\newtheorem{Prop}[Def]{Proposition}
\newtheorem{Lem}[Def]{Lemma}
\newtheorem{Thm}[Def]{Theorem}
\newtheorem{Cor}[Def]{Corollary}

\numberwithin{equation}{section}

\title{A Coupling of Brownian Motions In the $\mathcal{L}_{0}$-Geometry}
\author{Takafumi Amaba and Kazumasa Kuwada}
\thanks{
{\bf Mathematics Subject Classification(2010)}
Primary
53C21; 
Secondary
53C44, 
58J65, 
60J05, 
The first author was supported by JSPS KAKENHI Grant Number 24$\cdot$5772.
The second author is partially supported by the Grant-in-Aid for Young Scientist (B) 22740083 
}

\address{(T. Amaba) Ritsumeikan University, 1-1-1 Nojihigashi, Kusatsu, Shiga, 525-8577, Japan}
\email{fm-amaba@fc.ritsumei.ac.jp}

\address{(K. Kuwada) Tokyo Institute of Technology, 2-12-1 \^Ookayama, Meguro-ku, Tokyo 152-8551, Japan}
\email{kuwada@math.titech.ac.jp}

\date{}

\begin{document}
\maketitle

\begin{abstract}
Under a complete Ricci flow, we construct a coupling of two Brownian motion
such that their $\mathcal{L}_{0}$-distance is a supermartingale.
This recovers a result of Lott
[J. Lott, Optimal transport and Perelman's reduced volume,
Calc. Var. Partial Differential Equations 36 (2009), no. 1, 49--84.]
on the monotonicity of $\mathcal{L}_0$-distance between heat distributions. 
\end{abstract}

\section{Introduction} 
\label{sec:intro} 
Since Perelman's pioneering article \cite{Pe}, there are several attempts 
to study Ricci flow in connection with heat distributions. 
As one of them, J.~Lott \cite{Lo} provides several monotonicity formulae
from the viewpoint of optimal transportation 
by extending Topping's approach \cite{MT,To}. 
His argument is based on the Eulerian calculus, 
which can be rigorous if everything is sufficiently regular. 
On the other hand, 
sometimes it is not easy to verify the required regularity. 
For instance, we must take much care on it 
if the manifold is non-compact. 

Among results in \cite{Lo}, Lott introduced 
$\mathcal{L}_0$-functional on the space of (space-time) curves and 
the associated $\mathcal{L}_0$-distance as an analog of 
$\mathcal{L}$-distance in \cite{Pe,To}. 
He proved the monotonicity of transportation cost given in terms of 
$\mathcal{L}_0$-distance between heat distributions. 
By using it, he gave an alternative proof of the monotonicity of 
$\mathcal{F}$-functional in \cite{Pe}. 
The main purpose of this paper is to prove the monotonicity of 
$\mathcal{L}_0$-transportation cost by a probabilistic approach 
using a coupling of Brownian motions. 
As an advantage of our approach, we just require 
much weaker regularity assumptions and 
we can extend the result to more general transportation costs. 

The organization of this paper is as follows.
In the next subsection, 
we will state our framework and results more precisely. 
We also review Lott's result on the monotonicity of 
$\mathcal{L}_0$-transportation cost there. 
In subsection \ref{BackGround}, we give a review 
of historical background and related known results in more detail. 
All necessary calculations, formulae and properties 
on $\mathcal{L}_{0}$ are summarized in section \ref{L0Geom}. 
The reader, who wants to grasp
only heuristics or a rough story
of this paper, can skip section \ref{L0Geom} except for Proposition \ref{NonPos},
where we give a Hessian estimate for the $\mathcal{L}_{0}$-distance.
The most part of section \ref{L0Geom}
(e.g., the $\mathcal{L}_{0}$-cut locus, the $\mathcal{L}_{0}$-exponential and so on)
are analogous to ones for Riemannian distance
or $\mathcal{L}$-distance.

In sections \ref{AbsL0Cut} and \ref{PreL0Cut}, 
we will construct a coupling of $g(\tau)$-Brownian motions 
which satisfies a requirement of our main theorem 
(Theorem~\ref{Thm1} below). 
In section \ref{AbsL0Cut} we will discuss it 
under a strong regularity assumption on $\mathcal{L}_0$. 
More precisely, we assume that $\mathcal{L}_{0}$-cut locus is empty there 
(see section~\ref{L0Geom} for the definition of $\mathcal{L}_0$-cut locus; 
It is a concept analogous to the cut locus for the Riemannian distance function). 
While it is very restrictive, 
we believe that the argument given in that section will be insightful 
and that it helps us to understand the rigorous argument 
given in section \ref{PreL0Cut}. 
We employ an approximation of Brownian motions 
by geodesic random walks there, 
as in the previous result \cite{KP2} 
(see section~\ref{BackGround} for known results 
as well as the reason why we choose it).

\subsection{Framework and main results} 
\label{Framework} 

Let $(g(t))_{0\leq t \leq T}$ be a Ricci flow 
on a connected manifold $M$ without boundary with $d:=\dim M \geq 2$.
That is, $g(t)$ solves the following evolution equation: 
\begin{align} \label{RicFlow} 
\frac{\partial }{\partial t} g(t) = - 2 \Ric_{g(t)} . 
\end{align}
In the sequel, 
we always assume that $( M , g(t) )$ is 
a complete Riemannian manifold for each $t$. 

For stating Lott's result and ours, we introduce 
$\mathcal{L}_0$-functional and some notions concerning with it. 
Under the Ricci flow $g(t)$ on $d$-dimensional manifold $M$, 
$\mathcal{L}_{0}$-functional is given by 
\[
\mathcal{L}_{0} (\gamma) := \frac{1}{2} \int_{\T}^{\TT} \hspace{-3mm} \big\{
\vert \dot{\gamma}(t) \vert_{g(t)}^{2} + R_{g(t)}( \gamma (t) )
\big\} \D t
\]
for each piecewise smooth curve $\gamma :[\T ,\TT ] \to M$, 
where $R_{g(t)}$ is the scalar curvature with respect to $g(t)$. 
As a minimal value of $\mathcal{L}_0$-functional 
with the fixed endpoints (in space-time), 
we define $\mathcal{L}_{0}$-distance $L_0$. 
That is, for $0 \le \T < \TT \le T$ and $\M , \MM \in M$, 
$L_{0}^{\T,\TT} (\M, \MM)$ is given by
\[
L_{0}^{\T ,\TT} (\M ,\MM)
:=
\inf_{\gamma}
\mathcal{L}_{0} (\gamma),
\quad \M ,\MM \in M,
\]
where the infimum is taken over piecewise smooth curves 
$\gamma :[\T ,\TT] \to M$ such that 
$\gamma (\T) = \M$ and $\gamma (\TT) = \MM$.

We denote by $P(M)$ the space of Borel probability measures on $M$
and $P^{\infty}(M)$ the subspace of $P(M)$ whose element has smooth density.
For $0 \le \T < \TT \le T$ and $\mu^{\prime}$, $\mu^{\prime\prime} \in P(M)$, 
the (optimal) $\mathcal{L}_{0}$-transportation cost 
$C_{0}^{\T ,\TT} (\mu^{\prime},\mu^{\prime\prime})$
is defined by
\[
C_{0}^{\T ,\TT} (\mu^{\prime},\mu^{\prime\prime})
:=
\inf_{\pi \in \Pi (\mu^{\prime},\mu^{\prime\prime})}
\int_{M\times M} \hspace{-7mm} L_{0}^{\T ,\TT} (\M ,\MM) \pi (\D\M ,\D\MM ), 
\]
where $\Pi ( \mu^{\prime} , \mu^{\prime\prime} )$ is the set of 
couplings of $\mu^{\prime}$ and $\mu^{\prime\prime}$. 

In this framework, 
Lott proved the following: Assume that $M$ is closed. 
If 
$c^{\prime}$, $c^{\prime\prime}:[t_{0},t_{1}] \to P^{\infty}(M)$ are solutions to 
the backward heat equation 
%
\begin{align} \label{B-Heat} 
\frac{ \D \mu_{t} }{ \D t } = - \triangle_{g(t)} \mu_{t},
\end{align}
and if they satisfy some technical assumptions
(see Corollary 5 in \cite{Lo} for details, 
although we will mention it partially in the sequel),
then
\begin{align} \label{C0cost} 
u \mapsto
C_{0}^{\T + u, \TT +u}
\big(
c^{\prime} ( \T + u ) , c^{\prime\prime} ( \TT + u )
\big)
\end{align}
is non-decreasing (Proposition 13 in \cite{Lo}). 

Let us turn to state our result. 
Until the end of the paper, we fix two time intervals 
$$
0\leq \T_{0} < \T_{1} \leq T
\quad\text{and}\quad
0\leq \TT_{0} < \TT_{1} \leq T
$$
with $\T_{0} < \TT_{0}$ and $\T_{1} - \T_{0} = \TT_{1} - \TT_{0}=:S$.
We denote by $(\T , \TT)$ the coordinate on
$$
\big\{
(\T , \TT) \in [\T_{0},\T_{1}] \times [\TT_{0},\TT_{1}] : \T < \TT
\big\} .
$$

Since it looks awkward to work with backward heat
equation (\ref{B-Heat}), we shall reverse the time by setting
$$
\B = \B(s) := \T_{1} -s \quad\text{and}\quad \BB = \BB(s) := \TT_{1} -s
$$
for $0 \leq s \leq S$.
By {\it $g(\tau')$-Brownian motion (resp.~$g(\tau'')$-Brownian motion)}, 
we mean a time-inhomogeneous diffusion process on $M$ 
associated with $\triangle_{g(\tau'(s))}$ (resp.~$\triangle_{g(\tau''(s))}$), 
where $s$ stands for the time-parameter of the process. 

\begin{Thm} \label{Thm1} 
Assume that $(g(t))_{0\leq t\leq T}$ satisfies 
\begin{align} \label{CurAss} 
\inf_{(t,m) \in [ 0, T ] \times M} 
\inf_{
  \begin{subarray}{c} 
      V \in T_m M
      \\
      \| V \|_{g(t)} = 1 
  \end{subarray}    
} 
\Ric_{g(t)} (V,V) > - \infty .
\end{align}
Then for each $(\M ,\MM ) \in M \times M$, there exists a coupling of
$g(\tau^{\prime})$-Brownian motion $X=(X_{s})_{0\leq s \leq S}$
starting from $\M$ and
$g(\tau^{\prime\prime})$-Brownian motion $Y=(Y_{s})_{0\leq s \leq S}$
starting from $\MM$ such that
$$
s \mapsto L_{0}^{ \B(s) , \BB(s) } ( X_{s},Y_{s} )
$$
is a supermartingale and the map
$$
(m^{\prime}, m^{\prime\prime})
\mapsto
\text{the law of $(X,Y)$ with $( X_{0}, Y_{0} ) = ( m^{\prime}, m^{\prime\prime} )$}
$$
is measurable.
\end{Thm}

Theorem \ref{Thm1} provides us a probabilistic interpretation 
of the monotonicity of \eqref{C0cost}. 
That is, roughly speaking, we can show the monotonicity 
by taking an expectation of $L_{0}^{ \B(s) , \BB(s) } ( X_{s} , Y_{s} )$. 
Actually, we can say more: 
Let $\varphi :\mathbb{R} \to \mathbb{R}$ be concave and non-decreasing. 
We define a new transportation cost  
$
C_{0, \varphi}^{ t^{\prime}, t^{\prime\prime} }
( \mu^{\prime}, \mu^{\prime\prime} )
$
by
\[
C_{0, \varphi}^{ t^{\prime}, t^{\prime\prime} }
( \mu^{\prime}, \mu^{\prime\prime} )
:=
\inf_{ \pi \in \Pi ( \mu^{\prime}, \mu^{\prime\prime} ) }
\int_{ M \times M } \hspace{-5mm}
\varphi
\big(
	L_{0}^{ t^{\prime} , t^{\prime\prime} }
	( m^{\prime}, m^{\prime\prime} )
\big)
\pi ( \D m^{\prime}, \D m^{\prime\prime} ) .
\]

\begin{Cor} \label{Cor1}
Assume that our Ricci flow satisfies the condition \eqref{CurAss}.
Then for any two families $c(t^{\prime})$ and $c(t^{\prime\prime})$ of probability measures
satisfying \eqref{B-Heat},
$
C_{0, \varphi}^{ t^{\prime}+s, t^{\prime\prime}+s }
\big(
	c^{\prime}( t^{\prime} +s ),
	c^{\prime\prime}(t^{\prime\prime}+s)
\big)
$
is nondecreasing in $s$,
\end{Cor}

Note that Lott's result can be regarded as a special case of 
Corollary~\ref{Cor1}, that is, the case $\varphi (x)=x$ and $M$ is compact. 
Moreover, in order to make Otto's calculus rigorous, 
Lott further assumed that the curves $c^{\prime}$, $c^{\prime\prime}$ 
and the ``$E_{0}$-minimizing geodesics''
(see (4.9) in \cite{Lo} for the definition of $E_{0}$) 
which interpolate 
$c^{\prime}(t_{0}^{\prime}+s)$ 
and 
$c^{\prime\prime} (t_{0}^{\prime\prime}+s)$ 
lie in $P^{\infty}(M)$.
In \cite{Lo}, it is claimed that the last extra condition 
can be relaxed by giving an alternative proof which is analogous 
to Topping's approach in \cite{To}. 
In this paper, we give a proof 
based on the theory of stochastic calculus, 
and we do not only relax the extra regularity assumption 
but weaken the compactness assumption on $M$ 
to the curvature condition (\ref{CurAss}). 
Recall that, under a well-known sufficient condition to 
the unique existence of a Ricci flow in \cite{CZ,S2}, 
the condition \eqref{CurAss} is automatically satisfied. 
As a matter of fact, stochastic calculus is already used in \cite{KP2}
to extend Topping's result \cite{To} (see below for more details). 
Thus our result provides an additional evidence that 
stochastic calculus is an efficient tool to study this sort of problems.

\subsection{Historical background and related results}
\label{BackGround} 

Recall that the Ricci flow is a solution to 
the evolution equation \eqref{RicFlow}. 
This equation was introduced by Hamilton in \cite{H1} 
and used to find an Einstein metric 
(i.e., a metric $g$ such that $\Ric (g) = \mathrm{const.} g$) 
by deforming any given Riemannian metric $g_{0}$ with positive Ricci curvature 
on a compact $3$-manifold.

Inspired by quantum field theory, such as nonlinear $\sigma$-models, 
Perelman \cite{Pe} interprets the Ricci flow 
as a {\it gradient flow}; 
At least formally, the Ricci flow can be regarded 
as the gradient flow of the so-called 
Perelman's {\it $\mathcal{F}$-functional}.
This interpretation naturally leads 
to the {\it monotonicity formula} for $\mathcal{F}$: 
The functional $\mathcal{F}$ is nondecreasing along the Ricci flow. 
Additionally, the so-called {\it $\mathcal{W}$-functional} 
is also shown to be non-decreasing along the Ricci flow. 
As is well-known, these monotonicity formulae are effectively used 
in the resolution of Thurston's geometrization conjecture by Perelman.

Recently, alternative approaches to those formulae 
have been initiated on the basis of optimal transportation. 
For the monotonicity of $\mathcal{W}$, Topping \cite{To} gave 
an alternative proof by using the $\mathcal{L}$-transportation cost 
via so-called Lagrangian calculus. 
More precisely, he consider the optimal transportation cost whose cost function 
is given by a renormalization of Perelman's $\mathcal{L}$-distance. 
He proved the monotonicity of this transportation cost 
between (time-rescaled) heat distributions and 
derived the monotonicity of $\mathcal{W}$-functional 
by taking a sort of time-derivative of the optimal cost. 
For studying the monotonicity of $\mathcal{F}$, Lott \cite{Lo} showed 
the monotonicity of \eqref{C0cost} as explained in section~\ref{Framework}. 
Then he recovered the monotonicity of $\mathcal{F}$-functional 
(Corollary 6 in \cite{Lo}) again by taking a sort of time-derivative. 
As mentioned, Lott's argument is based on so-called Eulerian calculus 
(see e.g.~\cite{Vi} for a comparison with Lagrangian calculus). 

Such monotonicity formulae 
for optimal transportation costs 
between heat distributions are also studied in a slightly different context. 
For instance, on a complete Riemannian manifold with a fixed metric, 
the same sort of monotonicity of $L^p$-Wasserstein distance 
is equivalent to non-negative Ricci curvature or 
an $L^q$-gradient estimate for the heat semigroup 
(see \cite{RS} and references therein). 
For time-dependent metrics, 
the same sort of monotonicity of $L^2$-Wasserstein distance 
is shown to be equivalent to 
the property that the metric evolves as a super Ricci flow 
by McCann and Topping \cite{MT}. 
On one hand, the latter result is a natural extension of the former one 
since the latter recovers the former 
when the metric does not depend on time. 
On the other hand, this result can be regarded as a primitive form of 
the results \cite{Lo,To} 
(This observation is also addressed in \cite{Lo,To} themselves). 
These former results indicate that monotonicity formulae 
in the optimal transportation should be connected 
with the geometry of the space in a deeper way and 
that more studies are expected in this direction. 
\medskip

From its definition, optimal transportation cost is strongly 
related with the notion of coupling of random variables or stochastic processes. 
Thus it is natural to consider the above-mentioned problem 
by using a coupling method of stochastic processes. 
Even only in stochastic differential geometry, 
there are several results in coupling method. 
Traditionally, they paid much attention to estimating 
the time that two particle meets, 
while it does not match with our present purpose 
(see \cite{Cr,Ke,LiRo}; see \cite{Ke2,Wa} also). 
By using a similar idea as ones in those studies, 
we can construct a coupling by parallel transport 
on a complete Riemannian manifold with a lower Ricci curvature bound 
and it is tightly connected with the monotonicity of
Wasserstein distances (see \cite{RS} and references therein). 
Extensions of those kind of coupling 
to the time-dependent metric case are achieved 
by Kuwada \cite{Ku2} and by Arnaudon, Coulibaly and Thalmaier \cite{ACT}. 
A typical example of the time-dependent metric is backward (super) Ricci flow. 
To construct a coupling, 
the former used an approximation by coupled geodesic random walks 
and the latter construct a one-parameter family of coupled particles 
which consists of a string and moves continuously as time evolves. 
As a result, they recover the monotonicity formula in \cite{MT} 
and extend it to non-compact spaces. 
Topping's monotonicity formula is also proved and extended 
by Kuwada and Philipowski \cite{KP2} and 
discussed later by Cheng \cite{Ch}. 
The former used the same method as in \cite{Ku2} 
and the latter uses an argument studied in \cite{Wa}.

In all those couplings, 
we are interested in the behavior of distance-like functions 
(e.g.~distance itself or $\mathcal{L}$-distance) 
between a coupling of diffusion processes $( X_s , Y_s )$. 
In those cases, the main technical difficulty arises 
at singular points of the functions (e.g.~cut locus). 
Roughly speaking, there are two obstructions:
Firstly, the construction of the coupling itself 
depends on a regularity of the function. 
Secondly, we can not apply It\^o's formula directly 
when the coupled process lies on the singular points. 
Thus we require some indirect arguments as mentioned above 
to overcome these obstructions. 
In this paper, we follow the argument used in \cite{KP2,Ku2}. 
More precisely, we consider a coupling of geodesic random walks
$( X_{s}^{\varepsilon}, Y_{s}^{\varepsilon} )$. 
The construction of it requires less regularity 
and this fact works well to avoid the first obstruction. 
Then, instead of It\^o's formula for
$L_{0}^{ \tau^{\prime}(s), \tau^{\prime\prime}(s) } ( X_{s}, Y_{s} )$,
we can employ a {\it ``difference''} (rather than differential) {\it inequality} 
for $L_{0}^{ \tau^{\prime}(s), \tau^{\prime\prime}(s) } ( X_{s}^{\varepsilon}, Y_{s}^{\varepsilon} )$ 
at each approximation step $\varepsilon$ (see Proposition \ref{DiffIneq}), 
up to an well-controlled error. 
Along this idea, we can avoid the second obstruction. 
Since geodesic random walks converge to
(time-inhomogeneous) 
Brownian motions, 
we can obtain an estimate for a coupling of Brownian motions 
as the limit. 
For those who are interested in other approaches, 
it is worth mentioning that a comparison 
with other approaches is discussed in \cite{KP2,Ku2}.

\section{$\mathcal{L}_{0}$-geometry along Ricci flow}
\label{L0Geom} 

In the rest of the paper, we always assume \eqref{CurAss}. 

\subsection{
Differential calculus of $\mathcal{L}_0$ and $L_0$
}
The main aim in this subsection is
to give an estimate of (a contraction of) the Hessian for
$
L_{0}^{ t^{\prime}, t^{\prime\prime} }
$,
which we will use in the subsequent sections.
Let $\gamma :[\T ,\TT ] \to M$ be a piecewise smooth curve.
For each variation of $\gamma$ with a variational vector field $V$,
we denote by $( \delta_{V}\mathcal{L}_{0} ) (\gamma )$ and
$(\delta_{V} \delta_{V}\mathcal{L}_{0} ) (\gamma )$
the first and second variation of $\mathcal{L}_{0}$.
We omit all proofs in this subsection
except for Proposition \ref{NonPos},
because all proofs are routine
(see e.g., \cite[section 17, 18 and 19]{KL}).

\begin{Prop} \label{1stVar}
For any smooth variation
(not being necessarily proper)
of a smooth curve $\gamma : [t^{\prime}, t^{\prime\prime}] \to M$
with a variational vector field $V$,
we have
\begin{align*}
(\delta_{V}\mathcal{L}_{0}) (\gamma )
&=
\langle V(t), \dot{\gamma}(t) \rangle_{g(t)}
\Big\vert_{t=\T}^{t=\TT}
	+
\frac{1}{2} \int_{\T}^{\TT} \hspace{-2mm}
\langle
	\mathscr{G}_{t} ( \gamma ),
	V(t)
\rangle_{g(t)}
\D t
\end{align*}
where
$
\mathscr{G}_{t} ( \gamma )
:=
\nabla^{g(t)} R_{g(t)} - 2 \nabla_{\dot{\gamma} (t)}^{g(t)} \dot{\gamma} (t)
+ 4 \Ric_{g(t)} ( \dot{\gamma}(t), \cdot )
$.
In particular,
$
(\delta_{V}\mathcal{L}_{0}) (\gamma )
$
is independent of the choice of a variation which realizes
the variational vector field $V$ as its infinitesimal variation. 
Hence we call 
\begin{equation} \label{L0geod}
\nabla_{\dot{\gamma} (t)}^{g(t)} \dot{\gamma} (t)
- \frac{1}{2} \nabla^{g(t)} R_{g(t)}
-2 \Ric_{g(t)} ( \dot{\gamma}(t), \cdot )
= 0
\end{equation}
the $\mathcal{L}_{0}$-geodesic equation, where the $\Ric_{g(t)} ( \dot{\gamma}(t), \cdot )$ is naturally understood as
a $(1,0)$-tensor by the metric $g(t)$.
We call any solution of $\mathcal{L}_{0}$-geodesic equation
an $\mathcal{L}_{0}$-geodesic.
\end{Prop}

\begin{Prop} \label{1stDer} 
Assume that $L_{0}$ is smooth around
$
(
	t^{\prime} , m^{\prime};
	t^{\prime\prime} , m^{\prime\prime}
)
$
and that there exists a unique $\mathcal{L}_{0}$-minimizing
curve $\gamma$ joining
$( t^{\prime} , m^{\prime} )$
and
$( t^{\prime\prime} , m^{\prime\prime} )$.
Then we have
\begin{equation*}
\left\{\begin{array}{l}
\displaystyle
\frac{ \partial L_{0}^{ \T , \TT } }{ \partial \T } ( \M , \MM )
=
-\frac{1}{2} \Big\{
| \dot{\gamma} ( \T ) |_{g( \T )}^{2}
+
R_{g( \T )} ( \M )
\Big\}
-
\nabla_{ \dot{\gamma} (\T) }^{ g(\T) }
L_{0}^{\T , \TT} ( \cdot , \MM ), \vspace{2mm} \\
\displaystyle
\frac{ \partial L_{0}^{ \T , \TT } }{ \partial \TT } ( \M , \MM )
=
\frac{1}{2} \Big\{
| \dot{\gamma} ( \TT ) |_{g( \TT )}^{2}
+
R_{g( \TT )} ( \MM )
\Big\}
-
\nabla_{ \dot{\gamma} (\TT) }^{ g(\TT) }
L_{0}^{\T , \TT} ( \M , \cdot ) . 
\end{array}\right.
\end{equation*}
\end{Prop}

\begin{Prop} 
\label{1stDer-2}
Under the assumption in Proposition \ref{1stDer},
we have
\begin{equation*}
\nabla_{ m^{\prime} }^{ g(t^{\prime}) }
L_{0}^{ t^{\prime}, t^{\prime\prime} }
( \cdot , m^{\prime\prime} )
=
- \dot{\gamma} ( t^{\prime} )
\quad
\text{and}
\quad
\nabla_{ m^{\prime\prime} }^{ g(t^{\prime\prime}) }
L_{0}^{ t^{\prime}, t^{\prime\prime} }
( m^{\prime}, \cdot )
=
\dot{\gamma} ( t^{\prime\prime} ).
\end{equation*}
In particular, by combining with Proposition \ref{1stDer},
\begin{equation*}
\left\{\begin{array}{l}
\displaystyle
\frac{ \partial L_{0}^{\T , \TT} }{ \partial \T } ( \M , \MM )
=
\frac{1}{2} \Big\{
| \dot{\gamma} ( \T ) |_{ g(\T) }^{2} - R_{g(\T)} (\M)
\Big\}, \vspace{2mm} \\
\displaystyle
\frac{ \partial L_{0}^{\T , \TT} }{ \partial \TT } ( \M , \MM )
=
-\frac{1}{2} \Big\{
| \dot{\gamma} ( \TT ) |_{ g(\TT) }^{2} - R_{g(\TT)} (\MM)
\Big\}.
\end{array} \right.
\end{equation*}
\end{Prop} 

We denote the curvature tensor with respect to $g(t)$ by $\mathcal{R}_{g(t)}$. 
It appears in the following second variation formula for $\mathcal{L}_0$. 
For a piecewise smooth curve $\gamma : [ \T, \TT ] \to M$ and 
two vector fields $V$, $W$ along $\gamma$, we define the $\mathcal{L}_0$-index form 
$\mathcal{L}_0 I_\gamma ( V, W )$ as follows: 
\begin{align*}
\mathcal{L}_0 I_\gamma ( V , W ) 
: = 
\int_{\T}^{\TT} & 
\Big\{ 
  \langle 
    \nabla^{g(t)}_{\dot{\gamma} (t)} V(t) , 
    \nabla^{g(t)}_{\dot{\gamma} (t)} W(t) 
  \rangle_{g(t)} 
  + 
  \langle
    \mathcal{R}_{g(t)} ( V (t), \dot{\gamma} (t) ) W (t), 
    \dot{\gamma} (t)
  \rangle
\\
& +
  \frac12 \Hess R_{g(t)} ( V(t), W(t) )
  + 
  ( \nabla^{g(t)}_V \Ric_{g(t)} ) ( \dot{\gamma}(t) , W(t) ) 
\\ 
& + 
  ( \nabla^{g(t)}_W \Ric_{g(t)} ) ( \dot{\gamma}(t) , V(t) ) 
  - 
  ( \nabla^{g(t)}_{\dot{\gamma}(t)} \Ric_{g(t)} ( V(t) , W (t) )  
\Big\} \D t . 
\end{align*}
By definition, $\mathcal{L}_0 I_\gamma (V,W)$ is symmetric in $V$ and $W$.

\begin{Prop} 
\label{2ndVar} 
For any smooth variation
(not being necessarily proper)
of an $\mathcal{L}_{0}$-geodesic
$\gamma : [t^{\prime}, t^{\prime\prime}] \to M$
with a variational vector field $V$,
we have
\begin{equation}
\label{2ndVarform} 
\begin{split}
( \delta_{V} \delta_{V} \mathcal{L}_{0} ) ( \gamma )
&=
\left.
\big\langle
	\nabla_{V(t)}^{g(t)} V(t),
	\dot{\gamma} (t)
\big\rangle_{ g(t) }
\right\vert_{ t=t^{\prime} }^{ t=t^{\prime\prime} } 
+ \mathcal{L}_0 I_\gamma ( V , V ). 
\end{split}
\end{equation}
\end{Prop}

\begin{Rm} 
In (\ref{2ndVarform}), the second term
on the right hand side is independent of 
the choice of a variation of $\gamma$ which realizes 
the variational vector field $V$ 
as its infinitesimal variation. 
On the other hand, 
since $\gamma$ is $\mathcal{L}_{0}$-geodesic, 
the first term can be written as follows:
$$
\left.
\big\langle
	\nabla_{V(t)}^{g(t)} V(t),
	\dot{\gamma} (t)
\big\rangle_{ g(t) }
\right\vert_{ t=t^{\prime} }^{ t=t^{\prime\prime} }
=
( \delta_{ \nabla_{V}V } \mathcal{L}_{0} )
( \gamma ).
$$
\end{Rm} 

The next formula is derived from 
Proposition \ref{2ndVar}.

\begin{Prop} \label{Alt2Var}
Keeping the notations in Proposition \ref{2ndVar},
we have an alternative form of the second variational formula:
\begin{align*}
& ( \delta_{V} \delta_{V} \mathcal{L}_{0} ) ( \gamma ) \\
&=
\big\langle
	\nabla_{V(t)}^{g(t)} V(t),
	\dot{\gamma}(t)
\big\rangle \Big\vert_{t=\T}^{t=\TT}
+ \Ric_{g(t)} ( V(t), V(t) ) \Big\vert_{t=\T}^{t=\TT} \\
&\hspace{5mm}
+ \frac{1}{2} \int_{\T}^{\TT} \hspace{-3mm}
\Big\{
\Hess_{g(t)} R_{g(t)} (V(t),V(t))
+ 2 \Big\langle \mathcal{R}_{g(t)} ( V(t), \dot{\gamma}(t) ) V(t), \dot{\gamma}(t) \Big\rangle_{g(t)} \\
&\hspace{10mm}
- 2 \frac{\D \Ric_{g(t)}}{\D t} ( V(t),V(t) )
+ 4 \Big[
		\Big( \nabla_{V(t)}^{g(t)} \Ric_{g(t)} \Big) ( V(t),\dot{\gamma}(t) )
		- \Big( \nabla_{\dot{\gamma}(t)}^{g(t)} \Ric_{g(t)} \Big) (V(t),V(t))
	\Big] \\
&\hspace{15mm}
+ 2 \Big\vert
		\nabla_{\dot{\gamma}(t)}^{g(t)} V(t) - \Ric_{g(t)} (V(t),\cdot )
	\Big\vert_{g(t)}^{2}
- 2 |\Ric_{g(t)}(V(t), \cdot )|_{g(t)}^{2}
\Big\} \D t.
\end{align*}
\end{Prop}

To deduce an estimate of
(a contraction of)
the Hessian for $L_{0}^{ t^{\prime}, t^{\prime\prime} }$,
we need a testing vector field.
For this, we introduce the notion of space-time parallel transport.
This notion will be used also to construct
a coupling of two Brownian motions in subsequent sections.

\begin{Def}[Space-time parallel transport along an $\mathcal{L}_{0}$-minimizing curve]
\label{STP} 
Let $\M$, $\MM\in M$ and $\T < \TT$.
Let $\gamma$ be {\it a} $\mathcal{L}_{0}$-minimizing curve
joining $(t^{\prime}, m^{\prime})$ and $(t^{\prime\prime}, m^{\prime\prime})$.
We define the
{\it
space-time parallel transport
$\para_{\M ,\MM}^{\T ,\TT}:T_{\M}M \to T_{\MM}M$
along
$\gamma$
}
as
$$
\para_{\M ,\MM}^{\T ,\TT} (v) := V(\TT)
$$
by solving the linear differential equation
$$
\left\{\begin{array}{l}
\nabla_{ \dot{\gamma} (t) }^{g(t)} V(t) = \Ric_{g(t)} (V(t),\cdot ),
\quad \T \leq t \leq \TT , \\
V(\T ) = v.
\end{array}\right.
$$
One can check easily that $\para_{\M ,\MM}^{\T ,\TT}$ gives a linear isometry from
$(T_{\M}M, g(\T))$ to $(T_{\MM}M, g(\TT))$.
Note that the space-time parallel transport can be defined as an isometry 
for more general time-dependent metrics; see \cite[Remark~5]{KP2}.
\end{Def}

The main result in this subsection is the following.

\begin{Prop} \label{NonPos} 
Let $u_{1}^{\prime}, \cdots , u_{d}^{\prime}$
be an orthonormal basis
of $(T_{m^{\prime}}M, g(t^{\prime}))$.
Under the assumption in Proposition \ref{1stDer},
Then it holds that
\begin{align*}
&
\sum_{i=1}^{d}
\big[
	\Hess_{ g( t^{\prime} ) \oplus g( t^{\prime\prime} ) }
	L_{0}^{ t^{\prime} , t^{\prime\prime} }
\big]
\big(
	u_{i}^{\prime}
	\oplus
	\para_{ m^{\prime}, m^{\prime\prime} }^{ t^{\prime}, t^{\prime\prime} }
	u_{i}^{\prime} ,
	u_{i}^{\prime}
	\oplus
	\para_{ m^{\prime}, m^{\prime\prime} }^{ t^{\prime}, t^{\prime\prime} }
	u_{i}^{\prime}
\big) \\
&\hspace{30mm}
\leq
\Big\{
\frac{\partial L_{0}^{t' , t''}}{\partial t'}
+ \frac{\partial L_{0}^{t' , t''}}{\partial t''}
\Big\} 
( m^{\prime} , m^{\prime\prime} ) .
\end{align*}
\end{Prop}
 
For the proof, we gather formulae for geometric quantities along the Ricci flow.
For the proof, see \cite[equation (2.1.9) and subsection 2.5]{To1}.
\begin{Prop} \label{formula} 
Along the Ricci flow $\displaystyle \frac{\D g}{\D t}(t) = - 2 \Ric_{g(t)}$, one has
\begin{itemize}
\item[(i)] $\displaystyle
\frac{\partial R_{g(t)}}{\partial t} = \Delta_{g(t)} R_{g(t)} + 2|\Ric_{g(t)}|_{g(t)}^{2},
$

\vspace{2mm}
\item[(ii)] $\displaystyle
\mathrm{tr} \frac{ \D \Ric_{g(t)} }{\D t} = \Delta_{g(t)} R_{g(t)},
$

\vspace{2mm}
\item[(iii)] {\rm contracted Bianchi identity:}
$\displaystyle
\mathrm{tr} \Big( \nabla \Ric_{g(t)} \Big)
=
\frac{1}{2} \nabla^{g(t)} R_{g(t)}.
$
\end{itemize}
\end{Prop}

\begin{proof}[Proof of Proposition \ref{NonPos}]
Let $\gamma$ be an $\mathcal{L}_{0}$-minimizing curve
from $(t^{\prime},m^{\prime})$ to $(t^{\prime\prime},m^{\prime\prime})$.
By Proposition \ref{1stDer-2}, Proposition \ref{formula} and \eqref{L0geod}, 
we see
\begin{align*}
&
\Big\{
\frac{\partial L_{0}^{t' , t''}}{\partial t'}
+ \frac{\partial L_{0}^{t' , t''}}{\partial t''}
\Big\} 
( m^{\prime}, m^{\prime\prime} ) \\
&\hspace{5mm}=
\frac{1}{2}
\int_{ t^{\prime} }^{ t^{\prime\prime} } \hspace{-2mm} \Big\{
\Delta_{g(t)} R_{g(t)}
+ 2 |\Ric_{g(t)}|_{g(t)}^{2}
- 2 \Ric_{g(t)} ( \dot{\gamma}(t), \dot{\gamma}(t) )
\Big\} \D t .
\end{align*}

Next we compute and give an estimate for
(a contraction of)
the Hessian.
For each $i=1,2,\cdots ,d$,
we define a system of vector fields $(A_{i})_{i=1}^{d}$ along
$\gamma$ by
$$
A_{i} ( t )
:=
\para_{ m^{\prime}, \gamma (t) }^{ t^{\prime} , t }
u_{i}^{\prime}
\quad
\text{for $t^{\prime} \leq t \leq t^{\prime\prime}$}
$$
and we take a variation
$
f_{i} :
( -\varepsilon_{0} , \varepsilon_{0} )
\times
[ t^{\prime} , t^{\prime\prime} ]
\to
M
$
of $\gamma$
($\varepsilon_{0} >0$ being small enough)
such that
\begin{itemize}
\item[(a)] $f_{i}( 0 , \cdot ) = \gamma$,

\vspace{2mm}
\item[(b)] $f_{i}$ has $A_{i}$ as its variational field:
$$
A_{i} ( 0 , t )
=
A_{i} (t)
\quad
t^{\prime} \leq t \leq t^{\prime\prime}
$$
where
$\displaystyle
A_{i} ( \varepsilon , t )
:=
\frac{ \D f_{i} }{ \D \varepsilon } ( \varepsilon , t )
$
is the transversal vector field, and

\vspace{2mm}
\item[(c)] two transversal curves
$f_{i} ( \cdot , t^{\prime} )$,
$
f_{i} ( \cdot , t^{\prime\prime} ) :
( - \varepsilon_{0} , \varepsilon_{0} ) \to M
$
are $g( t^{\prime} )$-geodesic and $g( t^{\prime\prime} )$-geodesic respectively at $\varepsilon =0$:
$$
\nabla_{ A_{i} ( t^{\prime} ) }^{g( t^{\prime} )}
A_{ i } ( \cdot ,  t^{\prime} ) = 0
\quad
\text{and}
\quad
\nabla_{ A_{i} ( t^{\prime\prime} )}^{g( t^{\prime\prime} )}
A_{ i } ( \cdot , t^{\prime\prime} ) = 0 .
$$
\end{itemize}
We further set
$$
\ell_i (\varepsilon)
:=
L_{0}^{ t^{\prime} , t^{\prime\prime} }
\big(
	f_{i} ( \varepsilon , t^{\prime} ),
	f_{i} ( \varepsilon , t^{\prime\prime} )
\big)
\leq
\mathcal{L}_{0}
( f_{i} ( \varepsilon , \cdot ) ) =: \widehat{\ell}_i (\varepsilon).
$$
It is easy to see that
$
\ell_{i}^{\prime\prime}(0)
\leq
{ \widehat{\ell}_{i}\ }^{\hspace{-1.5mm}\prime\prime} \!\! (0)
$.
Since
$$
\nabla_{
	A_{i} ( t^{\prime} ) \oplus A_{i} ( t^{\prime\prime} )
}^{
	g( t^{\prime} ) \oplus g( t^{\prime\prime} )
}
A_{i} ( \cdot , t^{\prime} ) \oplus A_{i} ( \cdot , t^{\prime\prime} )
=
\big\{
\nabla_{
	A_{i} ( t^{\prime} )
}^{ g( t^{\prime} ) }
A_{i} ( \cdot , t^{\prime} )
\big\}
\oplus
\big\{
\nabla_{
	A_{i} ( t^{\prime\prime} )
}^{ g( t^{\prime\prime} ) }
A_{i} ( \cdot , t^{\prime\prime} )
\big\}
=
0, 
$$
we can compute the Hessian as
\begin{align*}
&
\big[
	\Hess_{ g( t^{\prime} ) \oplus g( t^{\prime\prime} ) }
	L_{0}^{ t^{\prime} , t^{\prime\prime} }
\big]
\big(
	u_{i}^{\prime}
	\oplus
	\para_{ m^{\prime} , m^{\prime\prime} }^{ t^{\prime}, t^{\prime\prime} }
	u_{i}^{\prime},
	u_{i}^{\prime}
	\oplus
	\para_{ m^{\prime} , m^{\prime\prime} }^{ t^{\prime}, t^{\prime\prime} }
	u_{i}^{\prime}
\big) \\
&=
\big[
	\Hess_{ g( t^{\prime} ) \oplus g( t^{\prime\prime} ) }
	L_{0}^{ t^{\prime} , t^{\prime\prime} }
\big]
\big(
	A_{i} ( t^{\prime} ) \oplus A_{i} ( t^{\prime\prime} ),
	A_{i} ( t^{\prime} ) \oplus A_{i} ( t^{\prime\prime} )
\big) \\
&=
A_{i} ( t^{\prime} ) \oplus A_{i} ( t^{\prime\prime} )
\big\{
	\big(
		A_{i} ( \cdot , t^{\prime} )
		\oplus
		A_{i} ( \cdot , t^{\prime\prime} )
	\big)
	L_{0}^{ t^{\prime} , t^{\prime\prime} }
\big\} \\
&=
\left.\frac{\D^{2}}{\D \varepsilon^{2}}\right\vert_{\varepsilon =0} \hspace{-5mm}
L_{0}^{ t^{\prime} , t^{\prime\prime} }
\big(
	f_{i} ( \varepsilon , t^{\prime} ),
	f_{i} ( \varepsilon , t^{\prime\prime} )
\big)
=
\ell^{\prime\prime}_{i} (0)
\leq
{ \widehat{\ell}_{i}\ }^{\hspace{-1.5mm}\prime\prime} \!\! (0)
=
( \delta_{A_{i}}\delta_{A_{i}} \mathcal{L}_{0} ) ( \gamma ).
\end{align*}
By the second variational formula 
(Proposition \ref{Alt2Var}), we have
\begin{align*}
&
( \delta_{A_{i}}\delta_{A_{i}} \mathcal{L}_{0} ) ( \gamma ) \\
&=
\Ric_{g(t)} ( A_{i} (t), A_{i} (t) )
\Big\vert_{ t = t^{\prime} }^{ t = t^{\prime\prime} } \\
&\hspace{5mm}
+ \frac{1}{2} \int_{ t^{\prime} }^{ t^{\prime\prime} } \hspace{-3mm}
\Big\{
	\Hess_{g(t)} R_{g(t)} ( A_i (t), A_i (t) ) \\
&\hspace{10mm}
	+
	2
	\Big\langle
		\mathcal{R}_{g(t)}
		( A_i (t), \dot{\gamma}(t) ) A_i (t), \dot{\gamma}(t)
	\Big\rangle_{g(\tau)}
	-
	2
	\frac{ \D \Ric_{g(t)} }{ \D t } ( A_i (t), A_i (t) ) \\
&\hspace{15mm}
	+
	4
	\Big[
		\Big( \nabla_{ A_i (t) }^{ g(t) } \Ric_{ g(t) } \Big)
		( A_i (t) , \dot{\gamma}(t) )
		-
		\Big( \nabla_{ \dot{\gamma}(t) }^{ g(t) } \Ric_{ g(t) } \Big)
		( A_i (t), A_i (t) )
	\Big] \\
&\hspace{10mm}
	+
	2
	\Big\vert
		\nabla_{ \dot{\gamma}(t) }^{ g(t) } A_i (t)
		-
		\Ric_{ g(t) } ( A_{i} (t), \cdot )
	\Big\vert_{ g(t) }^{2}
	-
	2
	| \Ric_{ g(t) } ( A_{i} (t), \cdot ) |_{ g(t) }^{2}
\Big\} \D t .
\end{align*}
Hence, by taking the sum over $i=1,2,\ldots ,d$
with formulae in Proposition \ref{formula},
we have
\begin{align*}
&
\sum_{i=1}^{d}
\big[
	\Hess_{ g( t^{\prime} ) \oplus g( t^{\prime\prime} )}
	L_{0}^{ t^{\prime} , t^{\prime\prime} }
\big]
\big(
	u_{i}^{\prime}
	\oplus
	\para_{ m^{\prime} , m^{\prime\prime} }^{ t^{\prime} , t^{\prime\prime} }
	u_{i}^{\prime} ,
	u_{i}^{\prime}
	\oplus
	\para_{ m^{\prime} , m^{\prime\prime} }^{ t^{\prime} , t^{\prime\prime} }
	u_{i}^{\prime}
\big) \\
&\leq
R_{g(t)} ( \gamma (t) )
\Big\vert_{ t=t^{\prime} }^{ t=t^{\prime\prime} } \\
&\hspace{5mm}
+
\frac{1}{2}
\int_{ t^{\prime} }^{ t^{\prime\prime} } \hspace{-3mm}
\Big\{
	\Delta_{g(t)} R_{g(t)}
	- 2 \Ric_{g(t)} ( \dot{\gamma}(t), \dot{\gamma}(t) )
	- 2 \Delta_{g(t)} R_{g(t)} \\
&\hspace{25mm}
+ 4
\Big[
	\frac{1}{2} \nabla_{ \dot{\gamma}(t) }^{g(t)} R_{g(t)}
	-
	\nabla_{ \dot{\gamma}(t) }^{g(t)} R_{g(t)}
\Big]
- 2 | \Ric_{g(t)} |_{g(t)}^{2}
\Big\} \D t \\
&=
\frac{1}{2}
\int_{ t^{\prime} }^{ t^{\prime\prime} } \hspace{-3mm}
\Big\{
	2 \Delta_{g(t)} R_{g(t)}
	+ 4 | \Ric_{g(t)} |_{g(t)}^{2}
	+ 2 \nabla_{ \dot{\gamma}(t) }^{g(t)} R_{g(t)}
\Big\} \D\tau \\
&\hspace{5mm}+
\frac{1}{2}
\int_{ t^{\prime} }^{ t^{\prime\prime} } \hspace{-3mm}
\Big\{
	- \Delta_{g(t)} R_{g(t)}
	- 2 \Ric_{g(t)} ( \dot{\gamma}(t), \dot{\gamma}(t) )
	- 2 \nabla_{ \dot{\gamma}(t) }^{g(t)} R_{g(t)}
	- 2 | \Ric_{g(t)} |_{g(t)}^{2}
\Big\} \D t \\
&=
\frac{1}{2}
\int_{ t^{\prime} }^{ t^{\prime\prime} } \hspace{-3mm}
\Big\{
	\Delta_{g(t)} R_{g(t)}
	+ 2 | \Ric_{g(t)} |_{g(t)}^{2}
	- 2 \Ric_{g(t)} ( \dot{\gamma}(t), \dot{\gamma}(t) )
\Big\} \D t ,
\end{align*}
which is equal to
$\displaystyle
\Big\{
\frac{\partial L_{0}^{t' , t''}}{\partial t'}
+
\frac{\partial L_{0}^{t' , t''}}{\partial t''}
\Big\} 
( m^{\prime}, m^{\prime\prime} )
$.
\end{proof}


\subsection{Some estimates on relatively compact open subsets}

We begin with estimates which hold \emph{globally} under
the curvature assumption \eqref{CurAss}. 
Let $K_- > 0$ be a constant satisfying
$-K_- g(t) \leq \Ric_{g(t)}$ for all $t \in [ 0 , T ]$. 
Recall that $\dim M = d$ and our Ricci flow is defined on $[0,T]$. 
Given any metric $g$, we denote by $\rho_{g}$ the corresponding
Riemannian distance.

\begin{Prop} \label{Est0} 
\quad \\
\begin{itemize}
\item[(i)] {\rm Comparison of metric $g(t)$ between two different times:}
$$
g(t) \leq \mathrm{e}^{ 2K_{-} (t-s) } g(s)
\quad
\text{for $0 \leq s \leq t \leq T$.}
$$

\item[(ii)] {\rm Comparison of distance $\rho_{g(t)}$ between two different times:}
$$
\rho_{g(t)}(x,y) \leq \mathrm{e}^{ K_{-}(t-s) } \rho_{g(s)} (x,y)
$$
for any $x$, $y\in M$ and $0 \leq s \leq t \leq T$.

\item[(iii)] {\rm Lower bound for $L_{0}$:}
For $0 \leq t^{\prime} < t^{\prime\prime} \leq T$,
$m^{\prime}$, $m^{\prime\prime} \in M$ 
and a piecewise $C^1$ curve $\gamma : [ \T , \TT ] \to M$ 
with $\gamma (\T) = \M$ and $\gamma (\TT) = \MM$,
\[
\mathcal{L}_0 (\gamma) 
\ge 
\frac12 \mathrm{e}^{-K_- (\TT - \T) } \int_{\T}^{\TT} | \dot{\gamma} (t) |_{g(\TT)}^2 \D t 
- \frac{d K_-}{2} ( \TT - \T ).
\]
In particular,
\begin{align*}
L_0^{\T, \TT} ( \M , \MM ) 
\ge 
\frac{1}{2} \mathrm{e}^{- K_- (\TT - \T)}
& 
\frac{ \rho_{g(\TT)} ( \M , \MM )^{2} }{ \TT - \T }
-
\frac{d K_-}{2} ( \TT - \T ) ,
\\
\inf_{\substack{(\T,\M ; \TT,\MM ) \in ( [0,T] \times M )^{2} \\ t^{\prime} < t^{\prime\prime}}} 
& 
L_{0}^{\T , \TT} ( \M , \MM )
\geq -\frac{dK_- T}{2}.
\end{align*}
\end{itemize}
\end{Prop}

Although the proof is parallel to
the corresponding statements in $\mathcal{L}$-geometry 
(especially (i) and (ii) are irrelevant to $\mathcal{L}$ or $\mathcal{L}_0$), 
we give it for completeness. 

\begin{proof} 
(i) Since
$\displaystyle
\frac{\partial g}{\partial t} = - 2\Ric_{g} \leq 2 K_{-} g
$,
apply Gronwall's lemma.

(ii) obviously follows from (i). 

(iii) Take a piecewise $C^1$-curve 
$
\gamma :[ t^{\prime} , t^{\prime\prime} ] \to M
$
with
$\gamma ( t^{\prime} ) = m^{\prime}$ and $\gamma ( t^{\prime} ) = m^{\prime\prime}$. 
Then, by the choice of $K_-$ and (ii), 
\begin{align*}
\mathcal{L}_0 (\gamma)
& \geq
\frac{1}{2}
\int_{ t^{\prime} }^{ t^{\prime\prime} }
\hspace{-2mm}
\vert \dot{\gamma}(t) \vert_{g(t)}^{2}
\mathrm{d}t
+
\frac{ d K_{-} }{2} \\
& \geq
\frac{1}{2}
\mathrm{e}^{ - K_{-} ( t^{\prime\prime} - t^{\prime} ) }
\int_{ t^{\prime} }^{ t^{\prime\prime} }
\hspace{-2mm}
\vert \dot{\gamma}(t) \vert_{g(t^{\prime\prime})}^{2}
\mathrm{d}t
+
\frac{ d K_{-} }{2} \\
& \geq
\frac{1}{2}
\mathrm{e}^{ - K_{-} ( t^{\prime\prime} - t^{\prime} ) }
\frac{
	\rho_{ g( t^{\prime\prime} ) } ( m^{\prime}, m^{\prime\prime} )^{2}
}{\TT - \T}
+
\frac{ d K_{-} }{2} .
\end{align*}
Thus the conclusion holds since $\gamma$ is arbitrary. 
\end{proof} 

\begin{Prop} 
\label{Est} 
Let
$\M , \MM \in M$,
$0 \leq \T < \TT \leq T$
and let
$\gamma : [ \T , \TT ] \to M$
be an $\mathcal{L}_0$-geodesic joining
$( t^{\prime}, m^{\prime} )$
to
$( t^{\prime\prime}, m^{\prime\prime} )$.
For each curve $\eta : [ t^{\prime}, t^{\prime\prime} ] \to M$, put
\begin{equation*}
\begin{split}
K_{+}( \eta )
&:=
\inf
\Big\{
	K>0 :
	\text{ $\Ric_{g(t)} \leq K g(t)$ along $\eta (t)$ }
\Big\}, \\
C( \eta )
&:=
\sup_{ t \in [t^{\prime}, t^{\prime\prime}] }
\vert \nabla^{g(t)} R_{g(t)} ( \eta (t) ) \vert_{g(t)}^{2} .
\end{split}
\end{equation*}
Then we have the following: 
\begin{itemize}
\item[(i)] {\rm Upper bound for $L_{0}$:}
For each $g(t^{\prime})$-geodesic $c : [ \T , \TT ] \to M$
with $c(t^{\prime}) = m^{\prime}$ and $c(t^{\prime\prime}) = m^{\prime\prime}$,
$$
L_{0}^{\T,\TT} (\M ,\MM)
\leq
\frac{ \rho_{g(\T)} ( \M ,\MM )^{2} }{ 2 (\TT -\T )^{2} }
\frac{
	\mathrm{e}^{ 2K_- ( \TT - \T ) }
	-
	1
}{ 2K_- }
+ \frac{ dK_{+} ( c ) (\TT -\T ) }{2}.
$$

\item[(ii)] {\rm Bound of $\displaystyle \frac{\D}{\D t} |\dot{\gamma} (t)|_{g(t)}^{2}$:}
For each $t \in [ t^{\prime}, t^{\prime\prime} ]$,
we have
$$
- \Big( 2K_- + \frac{1}{2} \Big)
|\dot{\gamma} (t)|_{g(t)}^{2} - \frac{ C (\gamma ) }{2}
\leq
\frac{\D}{\D t} | \dot{\gamma} (t)|_{g(t)}^{2}
\leq
\Big( 2K_{+} ( \gamma ) + \frac{1}{2} \Big) |\dot{\gamma} (t)|_{g(t)}^{2} + \frac{ C (\gamma ) }{2}.
$$

\vspace{2mm}
\item[(iii)] {\rm Comparison of $|\dot{\gamma} (t)|_{g(t)}^{2}$ between different times:}
There are constants $c_i > 0$ $(i=1,2,3,4)$ such that, 
for $\T \leq u^{\prime} \leq u^{\prime\prime} \leq \TT$,
\begin{align*}
|\dot{\gamma} (u^{\prime\prime})|_{g(u^{\prime\prime})}^{2}
&\leq
c_1 
|\dot{\gamma} (u^{\prime})|_{g(u^{\prime})}^{2}
+
c_2 , \\
|\dot{\gamma} (u^{\prime})|_{g(u^{\prime})}^{2}
&\leq
c_3 | \dot{\gamma} ( u^{\prime\prime} ) |_{g(u^{\prime\prime})}^{2}
+
c_4 
\end{align*}
where the constants depend only
on $T$, $K_-$, $K_{+}(\gamma )$ and $C (\gamma )$.

\item[(iv)] {\rm Bounding the speed of $\mathcal{L}_{0}$-minimizing curve at some time by $\mathcal{L}_{0}$:}
If the curve $\gamma$ is $\mathcal{L}_{0}$-minimizing,
there is $t^{*}\in (\T ,\TT )$ such that
$$
\frac12 
|\dot{\gamma} (t^{*})|_{g(t^{*})}^{2}
\leq
\frac{ L_{0}^{\T ,\TT} (\M ,\MM ) }{ \TT -\T } + \frac{dK_-}{2}.
$$
\end{itemize}
\end{Prop}

\begin{proof}
(i) The proof is similar to Proposition \ref{Est0}(iii).
By Proposition \ref{Est0}(i), 
we see that
\begin{align*}
L_{0}^{\T ,\TT}(\M ,\MM )
\leq \mathcal{L}_{0} (c)
\leq \frac{1}{2} \int_{\T}^{\TT} \hspace{-3mm}
\Big\{
	\mathrm{e}^{ 2K_- ( t - \T ) }
	\vert \dot{c} (t) \vert_{g(\T)}^{2}
	+
	d K_{+}(c)
\Big\} \D t .
\end{align*}
Since
$\displaystyle
\vert \dot{c} (t) \vert_{g(\T)}
\equiv
\frac{ \rho_{g(\T)} ( \M ,\MM ) }{ \TT - \T }
$,
the claim holds.

(ii) Using the
$\mathcal{L}_{0}$-geodesic equation,
\begin{align*}
&
\frac{\D}{\D t} \vert \dot{\gamma}(t) \vert_{g(t)}^{2}
=
- 2 \Ric_{g(t)} ( \dot{\gamma}(t), \dot{\gamma}(t) )
+
2
\big\langle
	\nabla_{\dot{\gamma}(t)}^{g(t)} \dot{\gamma}(t),
	\dot{\gamma}(t)
\big\rangle_{g(t)} \\
&=
2 \Ric_{g(t)} ( \dot{\gamma}(t), \dot{\gamma}(t) )
+
\big\langle
	\nabla^{g(t)} R_{g(t)},
	\dot{\gamma}(t)
\big\rangle_{g(t)} \\
&\leq
2 \Ric_{g(t)} ( \dot{\gamma}(t), \dot{\gamma}(t) )
+
\frac{
	\vert \nabla^{g(t)} R_{g(t)} \vert_{g(t)}^{2}
	+
	\vert \dot{\gamma}(t) \vert_{g(t)}^{2}
}{
	2
} \\
&\leq
\Big(
	2 K_{+} (\gamma )
	+
	\frac{1}{2}
\Big)
\vert \dot{\gamma}(t) \vert_{g(t)}^{2}
+
\frac{ C ( \gamma ) }{2} .
\end{align*}
The lower bound
$\displaystyle
\frac{\D}{\D t} \vert \dot{\gamma}(t) \vert_{g(t)}^{2}
\geq
-
\Big( 2K_- + \frac{1}{2} \Big)
\vert \dot{\gamma}(t) \vert_{g(t)}^{2}
-
\frac{ C (\gamma ) }{2}
$
can be obtained similarly.

(iii) By (ii), 
Gronwall's lemma implies
\begin{equation*}
\begin{split}
&
\vert \dot{\gamma} ( u^{\prime\prime} ) \vert_{ g( u^{\prime\prime} ) }^{2} \\
&\leq
\mathrm{e}^{
	( 2K_{+}(\gamma ) + \frac{1}{2} )
	( u^{\prime\prime} - u^{\prime} )
}
\vert
	\dot{\gamma} ( u^{\prime} )
\vert_{ g( u^{\prime} ) }^{2}
+
\frac{
	C (\gamma )
}{
	4K_{+} (\gamma ) +1
}
\Big(
	\mathrm{e}^{
		( 2K_{+} (\gamma ) +\frac{1}{2} )
		( u^{\prime\prime} - u^{\prime} )
	}
	- 1
\Big) .
\end{split}
\end{equation*}
The other is obtained similarly.

(iv) By the mean value theorem, 
we can take $t^{*} \in ( \T ,\TT )$ such that
$$
\vert \dot{\gamma} ( t^{*} ) \vert_{g(t^{*})}^{2}
=
\frac{1}{ \TT -\T } \int_{\T}^{\TT} \hspace{-3mm}
\vert \dot{\gamma} (t) \vert_{g(t)}^{2} \D t.
$$
Since $\gamma$ is $\mathcal{L}_{0}$-minimizing,
the right hand side is dominated by
\begin{align*}
\frac{ L_{0}^{\T ,\TT} (\M ,\MM ) }{ \TT - \T }
- \frac{1}{2 (\TT - \T ) } \int_{\T}^{\TT} \hspace{-3mm} R_{g(t)}
( \gamma (t) ) \D t
\leq
\frac{ L_{0}^{\T ,\TT} (\M ,\MM ) }{ \TT - \T } + 
\frac{ d K_- }{2}.
\end{align*}
\end{proof}

The following is a starting point of local estimates in this subsection. 
\begin{Lem} 
\label{Min} 
For each $\delta > 0$ and a relatively compact open subset $M_{0} \subset M$, 
there exists a relatively compact open subset
$B \supset M_{0}$ 
such that, for each $\M, \MM \in M_{0}$ and $t, \T, \TT \in [ 0, T ]$ with $\TT - \T \ge \delta$,
all $\mathcal{L}_{0}$-minimizing curves joining
$
( t^{\prime}, m^{\prime} )
$
and
$
( t^{\prime\prime}, m^{\prime\prime} )
$
and all $g(t)$-length-minimizing curves joining $\M$ and $\MM$ 
are contained in $B$.
\end{Lem} 

\begin{proof} 
Since $M_0$ is relatively compact, there is a compact $B_0$ 
with $M_0 \subset B_0 \subset M$ which contains 
all $g(0)$-geodesics joining any pair of points in $M_0$. 
Let us define $K_+ = K_+ ( B_0 )$ and $R_1 : = R_1 ( B_0 )$ by 
\begin{align*}
K_{+} ( B_0 )
&:=
\inf
\left\{
	K>0 :
	\text{$\Ric_{g(t)} \leq K g(t)$ on $B_0$}
\right\}, \\
R_1 ( B_0 )
&:=
\sup_{ \M , \MM \in M_0 }
\left\{ 
    \frac{ \rho_{g(0)} ( \M ,\MM )^{2} }{ 2 \delta^{2} }
    \frac{ \mathrm{e}^T - 1 }{ 2K_- }
    + 
    \frac{ dK_{+} ( B_0 ) T }{2} 
\right\}.
\end{align*}
Then, by Proposition \ref{Est}(i) and Proposition \ref{Est0}(ii), 
$L_0^{\T, \TT} ( \M , \MM ) \le R_1$ 
holds for all $0 \le \T < \TT \le T$ with $\TT - \T \ge \delta$ 
and $\M , \MM \in M_0$. 
Let $R_2 = R_2 (B_0)$ be defined as follows: 
\begin{equation*}
R_2 (B_0)
:=
2T \mathrm{e}^{ K_- T } \left( R_1 + \frac{d K_- T}{2} \right). 
\end{equation*}
Take a relatively compact open $B$ with $B_1 \subset B \subset  M$ 
such that $\rho_{g(T)} ( M_0 , B^c ) > \sqrt{R_2 / 2}$. 
Then, for any curve $\gamma : [ \T , \TT ] \to M$ 
with $\gamma ( \T ), \gamma ( \TT ) \in M_0$ and 
$\gamma ( [ \T , \TT ] ) \cap B^c \neq \emptyset$, 
Proposition \ref{Est0}(iii) and Proposition \ref{Est0}(i) 
yields $\mathcal{L}_0 (\gamma) > R_1$. 
Thus $B$ enjoys the claimed property on $\mathcal{L}_0$-minimizing curves. 
By a similar argument, we can prove the corresponding property 
for $g(t)$-geodesics by using Proposition \ref{Est0}(ii). 
Thus, the assertion holds by enlarging $B$ if necessary. 
\end{proof} 

\begin{Rm}
\label{PDS} 
By Proposition \ref{Est} and Lemma \ref{Min},
we see that
for each bounded open set $M_{0} \subset M$ and $\delta >0$,
there is positive constants
$K_{+} = K_{+} ( M_{0}, \delta ) > 0$ and $C = C( M_{0}, \delta )>0$
such that
\begin{equation*}
\Ric_{g} \leq K_{+} g,
\quad\text{and}\quad
\vert
	\nabla^{g(t)} R_{g(t)} (m)
\vert_{g(t)}^{2}
< C
\quad\text{on $[0,T] \times M_{0}$,}
\end{equation*}
and further, for each
$m^{\prime}$, $m^{\prime\prime} \in M_{0}$,
$t^{\prime\prime} - t^{\prime} \geq \delta$
and
each $\mathcal{L}_{0}$-minimizing curve $\gamma$ joining
$
( t^{\prime}, m^{\prime} )
$
and
$
( t^{\prime\prime}, m^{\prime\prime} )
$,
we have
\begin{equation}
\label{eq:L0Bound}
L_{0}^{ t^{\prime}, t^{\prime\prime} }
( m^{\prime}, m^{\prime\prime} )
\leq
\frac{
	\rho_{ g(t^{\prime}) }
	( m^{\prime}, m^{\prime\prime} )^{2}
}{
	2 ( t^{\prime\prime} - t^{\prime} )^{2}
}
\frac{
	\mathrm{e}^{ 2K_{-} t^{\prime\prime} }
	-
	\mathrm{e}^{ 2K_{-} t^{\prime} }
}{
	2 K_{-}
}
+
\frac{
	d K_{+} ( t^{\prime\prime} - t^{\prime} )
}{ 2 }
\end{equation}
and
\begin{equation}
\label{eq:speedBound}
-
\Big(
	2K_{-} + \frac{1}{2}
\Big)
\vert
	\dot{\gamma} (t)
\vert_{g(t)}^{2}
-
\frac{C}{2}
\leq
\frac{
	\mathrm{d}
}{
	\mathrm{d} t
}
\vert \dot{\gamma}(t) \vert_{g(t)}^{2}
\leq
\Big(
	2 K_{+} + \frac{1}{2}
\Big)
\vert \dot{\gamma}(t) \vert_{g(t)}^{2}
+
\frac{C}{2}.
\end{equation}
\end{Rm} 

\begin{Lem} 
\label{Continuity} 
$
(
	t^{\prime}, m^{\prime} ;
	t^{\prime\prime} , m^{\prime\prime}
)
\mapsto
L_0^{\T, \TT} ( m^{\prime} , m^{\prime\prime} )
$
is continuous on 
$\{ ( \T , \M ; \TT , \MM ) 
\; | \; 
0 \le \T < \TT \le T , \M , \MM \in M \}$.
\end{Lem} 
\begin{proof} 
Let $0 \le \T_0 < \TT_0 \le T$ and $\M_0 , \MM_0 \in M$ 
and take $\varepsilon \in ( 0, 1 )$. 
Take $\delta_0 > 0$ so that $\TT_0 - \T_0 \ge 2 \delta_0$. 
Let $U'$ and $U''$ be $g(0)$-metric balls of radius $\delta_0$ 
centered at $\M$ and $\MM$ respectively.
Let $B$ be as in Lemma \ref{Min} 
for $M_0 = U' \cup U''$ and $\delta = \delta_0$. 
Take $\delta_1' \in ( 0, \delta_0 / 4 )$ so that it satisfies the following: 
\begin{itemize}
\item
$( 1 - \varepsilon ) g (s) \le g(t) \le ( 1 + \varepsilon ) g(s)$ 
and $| R_{g(s)} - R_{g(t)} | < \varepsilon$ 
on $B$ for each $s , t \in [ 0 , T ]$ with $| s - t | < 2 \delta_1'$,  
\item
$4 \delta_1' / ( \delta_0 - 2 \delta_1' ) < \varepsilon$, 
\end{itemize}
Let $K_+$ and $C$ be as in Remark 2.2, corresponding to $M_0$ and $\delta_1'$. 
Take $\delta_1 \in ( 0 , \delta_1' )$ so that 
$d K_+ \delta_1 /2 < \varepsilon$. 
Take $\delta_2 > 0$ so that 
\[
\frac{\delta_2^2 
  ( e^{4 K_- T} 
  - e^{2 K_- ( 2 T - \delta_1 )} )
}{4 \delta_1^2 K_-} 
< \varepsilon 
\]
holds and take $\delta_3 = \delta_1 \wedge \delta_2$. 
Let $\M , \MM \in M$ with 
$\rho_{g(0)} ( \M , \M_0 ) \vee \rho_{g(0)} ( \MM , \MM_0 ) < \delta_3$ 
and $\T , \TT \in [ 0 , T ]$ with $| \T - \T_0 | \vee | \TT - \TT_0 | < \delta_3$. 
Take a curve $\gamma : [ \T_0 , \TT_0 ] \to M$ 
from $( \T_0 , \M_0 )$ to $( \TT_0 , \MM_0 )$ such that 
$\mathcal{L}_0 (\gamma) \le L^{\T_0 , \TT_0} ( \M_0 , \MM_0 ) + \varepsilon$. 
In addition, we take a curve $\gamma_1$ 
from $( \T, \M )$ to $( \T + \delta_1 , \M_0 )$, 
and a curve $\gamma_2$ 
from $( \TT - \delta_1 , \MM_0 )$ 
to $( \TT , \MM )$. 
Let $\alpha : [ \T + \delta_1 , \TT - \delta_1 ] \to [ \T_0 , \TT_0 ]$ 
be a (unique) affine increasing surjection. 
We define $\tilde{\gamma} : [ \T , \TT ] \to M$ as follows: 
\begin{equation*}
\tilde{\gamma} (t) : = 
\begin{cases}
\gamma_1 (t) & (t \in [ \T , \T + \delta_1 )), 
\vspace{.5ex}
\\
\displaystyle
\gamma 
\left( \alpha (t) \right)
& (t \in [ \T + \delta_1 , \TT - \delta_1 )), 
\vspace{.5ex}
\\
\gamma_2 (t) & (t \in [ \TT - \delta_1 , \TT ]).
\end{cases}
\end{equation*}
Note that, by the choice of $\delta_1$, 
\begin{align}
\label{eq:scale1}
\alpha' (t) 
& = 
\frac{ \TT_0 - \T_0 }{ \TT - \T - 2 \delta_1} \in [ 1 , 1 + \varepsilon ], 
\\
\label{eq:scale2}
| \alpha^{-1} (t) - t | 
& = 
\left| 
    \frac{ \TT - \T - 2 \delta_1}{ \TT_0 - \T_0 } 
    ( t - \T_0 ) 
    - 
    ( t - \T - \delta_1 )
\right| 
\\ \nonumber
& \le 
| \T_0 - \T - \delta_1 | 
\vee | \TT - \TT_0 - \delta_1 | \le 2 \delta_1 .
\end{align}

Now we turn to the estimate. 
We begin with the following basic estimate: 
\begin{align} \label{eq:tri}
L_0^{\T , \TT} ( \M , \MM ) - L_0^{\T_0 , \TT_0} ( \M_0 , \MM_0 ) 
& \le 
\mathcal{L}_0 ( \gamma_1 ) + \mathcal{L}_0 ( \gamma_2 ) 
\\ \nonumber
& \quad + 
\mathcal{L}_0 ( \tilde{\gamma}|_{[ \T + \delta_1 , \TT - \delta_1 ]} ) 
- \mathcal{L}_0 ( \gamma ) + \varepsilon . 
\end{align}
By the choice of $\tilde{\gamma}$, we have 
\begin{align*}
\mathcal{L}_0 ( \tilde{\gamma}|_{[ \T + \delta_1 , \TT - \delta_1 ]} ) 
& - \mathcal{L}_0 ( \gamma )
\\
& =  
\frac12 \int_{\T_0}^{\TT_0} 
\left\{
    \frac{ \TT_0 - \T_0 }{ \TT - \T - 2 \delta_1} 
    | \dot{\gamma} (t) |_{g ( \alpha^{-1} (t) )}^2 
    - 
    | \dot{\gamma} (t) |_{g(t)}^2 
\right\} 
\D t 
\\
& \quad + \frac12 
\int_{\T_0}^{\TT_0} 
\left\{ 
    \frac{ \TT - \T - 2 \delta_1}{ \TT_0 - \T_0 } 
    R_{g(\alpha^{-1} (t))} 
    - 
    R_{ g (t) } 
\right\} 
\D t .
\end{align*}
Then the choice of $\delta_1$ together with 
\eqref{eq:speedBound},
\eqref{eq:scale1} and \eqref{eq:scale2} 
yields that there is a constant $\hat{C} > 0$ depending only 
on $T$, $K_\pm$, $d$ and $C$
such that the right hand side of the last equality is bounded from above 
by $\hat{C} \varepsilon$. 
Moreover, by virtue of the choice of $\delta_1$ and $\delta_2$, 
\eqref{eq:L0Bound} yields  
$
  L_0^{\T, \T + \delta_1} ( \M , \M_0 ) 
  \vee 
  L_0^{\TT - \delta_1  , \TT} ( \MM_0 , \MM ) 
\le 
2 \varepsilon
$.
Thus, by minimizing the right hand side of \eqref{eq:tri} 
over $\gamma_1$ and $\gamma_2$, 
the left hand side is bounded from above 
by $( 5 + \hat{C} ) \varepsilon$. 
We can give the same upper bound to 
$L_0^{\T_0 , \TT_0} ( \M_0 , \MM_0 ) - L_0^{\T , \TT} ( \M , \MM )$ 
in the same manner and hence the assertion holds. 
\end{proof} 

We fix a bounded open $M_{0} \subset M$ and $\delta >0$ arbitrarily
and denote by
$
K_{+} = K_{+} ( M_{0}, \delta )
$
and
$
C = C ( M_{0}, \delta )
$
the positive constants appeared in Remark \ref{PDS}.
Let
$$
\mathcal{M}_{0}
:=
\Big\{
( \T ,\M ;\TT ,\MM )
\in
[0, T] \times M_{0} \times [0, T ] \times M_{0}
: \T < \TT
\Big\} .
$$
For each $(\T ,\M ; \TT ,\MM ) \in \mathcal{M}_{0}$,
we denote by $\Gamma_{\M ,\MM}^{\T ,\TT}$
the set of all $\mathcal{L}_{0}$-minimizing paths joining $(\T ,\M)$ to $(\TT ,\MM )$.
We further define
\begin{align*}
\mathcal{M}_{\delta}
&:=
\big\{
	( \T ,\M ;\TT ,\MM )
	\in
	\mathcal{M}_{0} :
	\TT - \T > \delta
\big\}, \\
\Delta_{\delta}
&:=
\Big\{
	(\T ,\TT )
	\in
	[ 0, T ] \times [ 0, T ] :
	\TT - \T > \delta
\Big\}, 
\\
\Gamma_{\delta}
&:= \hspace{-3mm}
\bigcup_{
	(\T ,\M ;\TT ,\MM ) \in \overline{\mathcal{M}_{\delta}}
} 
\hspace{-7mm}
\Gamma_{\M ,\MM}^{\T ,\TT} \quad \text{and} \\
\iota : & \Gamma_{\delta} 
\to \overline{\Delta}_{\delta} \times C ([0,1] \to M)
\\
&\hspace{30mm} \text{by}\quad
\iota (\gamma ) := ( \T ,\TT , \widehat{\gamma} )
\quad \text{if $\gamma \in \Gamma_{\M ,\MM}^{\T ,\TT}$}
\end{align*}
where $\widehat{\gamma} \in C ( [0,1] \to M )$
is defined by
$$
\widehat{\gamma} (u)
:=
\gamma \big( \T + u (\TT - \T ) \big)
\quad \text{for $0\leq u \leq 1$.}
$$
We topologize $\Gamma_{\delta}$ by the pull-back distance
$$
d_{ \Gamma_{\delta} } (\gamma ,\eta)
:=
|\T - s^{\prime}| + |\TT - s^{\prime\prime}|
+ \sup_{0\leq u \leq 1} \rho_{g(T)} ( \widehat{\gamma}(u), \widehat{\eta}(u) )
$$
for $\gamma \in \Gamma_{\M ,\MM}^{\T ,\TT}$ and
$\eta \in \Gamma_{n^{\prime} ,n^{\prime\prime}}^{s^{\prime},s^{\prime\prime}}$.

\begin{Lem}\label{Speed} 
There are constant $c_5, c_6 > 0$ such that, for
each $0 \leq \T < \TT \leq T$, $\M$, $\MM \in \overline{M_{0}}$ 
and
$\gamma \in \Gamma_{\M ,\MM}^{\T ,\TT}$,
\begin{align*}
\sup_{\T \leq t \leq \TT} \vert \dot{\gamma} (t) \vert_{g(t)}^{2}
\leq
2 c_5
\Big\{
\frac{
	L_{0}^{ t^{\prime}, t^{\prime\prime} }
	( m^{\prime}, m^{\prime\prime} )
}{ \TT - \T }
+
\frac{ d K_{-} }{2}
\Big\}
+
c_{6}.
\end{align*}
\end{Lem}
\begin{proof}
For each $0 \le \T <  \TT \le T$, 
$\M , \MM \in \overline{M_0}$ and $\gamma \in \Gamma_{\M, \MM}^{\T , \TT}$, 
the claimed bound follows from Proposition \ref{Est}(iii) and (iv)
with constants $\max \{ c_1 , c_3 \}$ and $\max \{ c_2 , c_4 \}$ as $c_5$ and $c_6$. 
By virtue of Lemma \ref{Min} (cf. Remark \ref{PDS}), 
we can choose them to be independent of 
$\T , \TT$, $\M , \MM$ and $\gamma$.
\end{proof}

\begin{Prop}[Compactness result]\label{Cpt} 
We have the following: 
\begin{itemize}
\item[(i)] {\rm Equi-Lipschitz estimate:}
For each $0 \leq \T < \TT \leq T$, 
$\M$, $\MM \in \overline{M_0}$,  
$\gamma \in \Gamma_{\M ,\MM}^{\T ,\TT}$
and $\T \leq s \leq u \leq \TT$,
\begin{align*}
&
\rho_{g(T)} ( \gamma (s), \gamma (u) )
\leq
\mathrm{const.} \Big( \frac{ 1 }{ \TT - \T } + 1 \Big) \vert u-s \vert , 
\end{align*}
where the constant depends on
$K_-$, $K_+$, $C$ 
but not on $s$, $u$, $\T$, $\TT$, $\M$, $\MM$
and the choice of $\gamma$.

\vspace{2mm}
\item[(ii)] {\rm Uniform boundedness:} $\Gamma_{\delta}$ is uniformly bounded.

\item[(iii)] {\rm Closedness:} $\iota (\Gamma_{\delta} )$ is closed in
$\overline{\Delta}_{\delta} \times C([0,1] \to M)$, 
where the topology of $\overline{\Delta}_{\delta} \times C([0,1] \to M)$ 
is given by the product of the Euclidean one in $\overline{\Delta}_{\delta}$ 
and the uniform topology in $C^{1}([0,1] \to M)$. 

\end{itemize}
Moreover, $\Gamma_{\delta}$ is compact.
\end{Prop}
\begin{proof}
Equi-Lipschitz estimate:
For $0 \le s < u \le T$, we have
\begin{align*}
\rho_{g(T)} \big( \gamma (s), \gamma (u) \big)^{2}
\leq
\int_{s}^{u} \hspace{-2mm}
\vert \dot{\gamma} (t) \vert_{g(T)}^{2} \D t
\leq
\mathrm{e}^{ 2 K_{-} T }
\int_{s}^{u} \hspace{-2mm}
\vert \dot{\gamma} (t) \vert_{g(t)}^{2} \D t ,
\end{align*}
where we have used Proposition \ref{Est0}(i) in the last inequality.
Then the equi-Lipschitz estimate follows by
Lemma \ref{Speed} and the compactness of $\overline{M_{0}}$.

Uniform boundedness: It is obvious from the Equi-Lipschitz estimate.

Closedness:
Define
$
\mathscr{L}_{0}:
\overline{\Delta}_{\delta} \times C^{1} ( [0,1] \to M )
\to (-\infty , +\infty ]
$
by
$$
\mathscr{L}_{0} (\T ,\TT ,c)
:=
\frac{1}{2}
\int_{\T}^{\TT} \hspace{-3mm}
\Big\{
	\vert \dot{\widetilde{c}} (t) \vert_{g(t)}^{2}
	+
	R_{g(t)} ( \widetilde{c} (t) )
\Big\} \D t
$$
where
$\displaystyle
\widetilde{c} (t) := c \Big( \frac{t-\T}{\TT -\T} \Big)
$
for $\T \leq t \leq \TT$. 
Then $\mathscr{L}_{0}$ is lower semicontinuous
(see e.g., \cite[Appendix 1, pp198--201]{Ma}).
By using $\mathscr{L}_0$, $\iota ( \Gamma_\delta )$ is expressed as follows: 
$$
\iota ( \Gamma_{\delta} )
=
\Big\{
(\T , \TT ,c) :
\begin{array}{l}
\text{$c(0)$, $c(1) \in \overline{M}$, $\TT - \T \geq \delta$} \\
\text{ and $\mathscr{L}_{0}(\T ,\TT ,c) \leq L_{0}^{\T ,\TT} (c(0),c(1))$ }
\end{array}
\Big\}
$$
This expression yields that $\iota ( \Gamma_\delta )$ is closed
in
$
\overline{\Delta}_{\delta} \times C( [0,1] \to M )
$.

Finally, by combining them with Ascoli's compactness theorem, 
we conclude that $\Gamma_\delta$ is compact. 
\end{proof}

\subsection{The $\mathcal{L}_{0}$-cut locus}
\label{sec:L0cut}

Set
$$
\mathcal{M}
:=
\Big\{
(\T ,\M ; \TT , \MM ) \in [0, T] \times M \times [0, T] \times M
: \T < \TT
\Big\} .
$$

\begin{Def}[$\mathcal{L}_{0}$-exponential map and $\mathcal{L}_{0}$-cut locus]\ \\
\begin{itemize}
\vspace{-4mm}
\item[(i)] $\mathcal{L}_{0}\exp_{\M}^{\T,\TT} : T_{\M}M \to M$ is defined by
$$
\mathcal{L}_{0} \exp_{\M}^{\T ,\TT} (v) := \gamma (\TT)
$$
where the curve $\gamma$ is the solution to the $\mathcal{L}_{0}$-geodesic equation
$$
\left\{\begin{array}{l}
\displaystyle
\nabla_{\dot{\gamma}(t)}^{g(t)} \dot{\gamma}(t)
- \frac{1}{2} \nabla^{g(t)} R_{g(t)}
- 2 \Ric_{g(t)} ( \dot{\gamma}(t), \cdot )
= 0, \\
\gamma (\T) = \M ,\quad \dot{\gamma}(\T) = v \in T_{\M} M.
\end{array}\right.
$$
See the Remark \ref{WelDef}(i) below.

\vspace{2mm}
\item[(ii)] $\mathcal{L}_{0}\mathrm{cut}_{\M}^{\T ,\TT}$ is the set of all points $\MM \in M$
such that there is more than one $\mathcal{L}_{0}$-minimizing curve
$\gamma :[\T ,\TT] \to M$ with $\gamma (\T) = \M$ and $\gamma (\TT) = \MM$
or there is a $v \in T_{\M}M$ such that
$\MM = \mathcal{L}_{0} \exp_{\M}^{\T ,\TT} (v)$ and $v$ is a critical point of
$\mathcal{L}_{0}\exp_{\M}^{\T ,\TT}$.
We also use the notation
$
\mathcal{L}_{0}\mathrm{cut}^{\T ,\TT}
:=
\{
	( m^{\prime}, m^{\prime\prime} ) :
	m^{\prime\prime} \in \mathcal{L}_{0}\mathrm{cut}_{\M}^{\T ,\TT}
\}
$.

\vspace{2mm}
\item[(iii)] $\displaystyle
\mathcal{L}_{0}\mathrm{cut}
:= \Big\{
( \T ,\M ; \TT , \MM ) \in \mathcal{M} : \MM \in \mathcal{L}_{0}\mathrm{cut}_{\M}^{\T ,\TT}
\Big\} .
$
\end{itemize}
\end{Def}

\begin{Rm} 
\label{WelDef} 
(i)
Since our Ricci flow is assumed to be complete,
we see that $\mathcal{L}_{0}\exp_{\M}^{\T,\TT}$ is well-defined as follows.
Given initial data
$\gamma (t^{\prime}) = m^{\prime}$ and $\dot{\gamma} (t^{\prime}) = v$,
let $I \subset [ t^{\prime} , T]$ be the maximal interval,
on which the $\mathcal{L}_{0}$-geodesic equation for $\gamma$ can be solved
(Recall that our Ricci flow is defined on $[0,T]$).
Since the $\mathcal{L}_{0}$-geodesic equation is of the normal form,
the standard theory of ODE shows that $I$ is open in $[ t^{\prime} , T]$.
On the other hand, the Gronwall lemma,
applied to the first inequality in Proposition \ref{Est}(iii)
(take $u^{\prime}=t^{\prime}$ and $u^{\prime\prime}=t$),
gives an upper bound for $\sup_{t \in I} \vert \dot{\gamma}(t) \vert_{g(t)}$.
This implies that $\{ \gamma (t) : t \in I \}$ is a bounded set.
From the completeness of the metric, $\gamma (t)$ converges as
$t \to \sup I$, which turns to show that $I$ is closed in $[ t^{\prime} , T]$.
Since $[ t^{\prime}, T]$ is connected and $I \neq \emptyset$,
we must have $I = [ t^{\prime}, T]$.
As $v$ is arbitrary, we conclude that $\mathcal{L}_{0}\exp_{\M}^{\T,\TT}$ is defined 
on the whole of $T_{\M}M$.

(ii)
For any $\mathcal{L}_{0}$-geodesic
$
\gamma : [ t^{\prime} , t^{\prime\prime} ] \to M
$, 
if $\dot{\gamma} ( t^{\prime} ) \neq 0$,
then 
$\dot{\gamma} ( t^{\prime\prime} ) \neq 0$ 
must hold.
This can be seen as follows:
Let
$\eta (t) := \gamma ( t^{\prime\prime} + t^{\prime} - t )$
and
$\widetilde{g} (t) := g ( t^{\prime\prime} + t^{\prime} - t )$
for $t \in [ t^{\prime} , t^{\prime\prime} ]$.
Then $\eta$ satisfies the differential equation
$$
\nabla_{ \dot{\eta}(t) }^{ \widetilde{g}(t) } \dot{\eta}(t)
- \frac{1}{2} \nabla^{ \widetilde{g}(t) } R_{ \widetilde{g}(t)}
+ 2 \Ric_{ \widetilde{g}(t) } ( \dot{\eta}(t), \cdot )
= 0
$$
with
$\eta ( t^{\prime} ) = \gamma ( t^{\prime\prime} )$
and
$\dot{\eta} ( t^{\prime} ) = - \dot{\gamma} ( t^{\prime\prime} )$.
If $\dot{\gamma} ( t^{\prime\prime} ) = 0$,
then
$
\nabla^{ \widetilde{g}(t^{\prime}) } R_{ \widetilde{g}(t^{\prime}) }
( \eta ( t^{\prime} ) )
=
\nabla^{g(t^{\prime\prime})} R_{g(t^{\prime\prime})}
( \gamma ( t^{\prime\prime} ) )
= 0
$.
Therefore $\eta^{*} (t) \equiv \eta ( t^{\prime} )$
also satisfies the same differential equation as $\eta$
with the same initial condition.
By uniqueness, 
$\eta^{*} (t) \equiv \eta ( t )$ must hold and hence 
$\dot{\gamma} (t^{\prime}) = -\dot{\eta}( t^{ \prime \prime} ) = 0$.
Thus the assertion holds.
\end{Rm} 

\begin{Prop} 
\label{CutLoc}\ \\
\begin{itemize}
\vspace{-4mm}
\item[(i)]
For each $\T <\TT$ and $\M$, $\MM \in M$, there is a smooth path
$\gamma :[\T ,\TT] \to M$ joining $\M$ to $\MM$ such that $\gamma$ has the
minimal $\mathcal{L}_{0}$-length among all such paths 
(see Lemma 7.27 in \cite{CCGGIIKLLN}).

\vspace{2mm}
\item[(ii)]
For any $\M \in M$ and $\T < \TT$,
$\mathcal{L}_{0}\mathrm{cut}_{\M}^{\T ,\TT}$
is closed and of measure zero 
(see Lemma 7.99 in \cite{CCGGIIKLLN} and Lemma 5 in \cite{MT}).

\vspace{2mm}
\item[(iii)]
The $\mathcal{L}_{0}\mathrm{cut}$ is closed.

\vspace{2mm}
\item[(iv)]
The function $L_{0}^{\T,\TT}(\M ,\MM)$ is smooth
on $\mathcal{M} \backslash \mathcal{L}_{0}\mathrm{cut}$.
\end{itemize}
\end{Prop}

\begin{Rm} 
By Proposition \ref{1stVar},
any $\mathcal{L}_0$-minimizing curve must be $\mathcal{L}_0$-geodesic 
and then Proposition \ref{CutLoc}(i) shows that $\exp_{\M}^{\T,\TT}$ is surjective.
Additionally, by (iv),
the statements in Proposition \ref{1stDer}, \ref{1stDer-2} and \ref{NonPos}
hold outside $\mathcal{L}_{0}\mathrm{cut}$.
\end{Rm} 

Since they can be shown by the same arguments as the usual Riemannian geometry
or $\mathcal{L}$-geometry, we omit the proof of (i) and (iv). 
The proof of (iii) is along the same line of Lemma 5 in \cite{MT}.

\begin{proof}[{\bf Proof of Proposition \ref{CutLoc} (ii)}]
Since the closedness follows from (iii),
we prove that $\mathcal{L}_{0}\mathrm{cut}_{\M}^{\T ,\TT}$
is of measure zero. First we decompose
$
\mathcal{L}_{0}\mathrm{cut}_{\M}^{\T ,\TT}
=
\mathcal{C} \cup \mathcal{C}^{\prime}
$
where $\mathcal{C}$ is the set of all critical values in
$\mathcal{L}_{0}\mathrm{cut}_{\M}^{\T ,\TT}$
of $\mathcal{L}_{0}\exp_{\M}^{\T ,\TT}$ and
$$
\mathcal{C}^{\prime}
:=
\Big\{
\MM \in \mathcal{L}_{0}\mathrm{cut}_{\M}^{\T ,\TT}:
\begin{array}{l}
\text{$\MM$ is a regular value of $\mathcal{L}_{0}\exp_{\M}^{\T ,\TT}$ and} \\
\text{there is more than one $\mathcal{L}_{0}$-minimizing} \\
\text{curve joining $(\T ,\M)$ and $(\TT ,\MM)$.}
\end{array}
\Big\} .
$$
By Sard's theorem, $\mathcal{C}$ has measure zero and hence we need only to prove
so is also $\mathcal{C}^{\prime}$. For this, consider the map
$\phi :T_{\M} M \times T_{\M} M \to \mathbb{R}$
defined by
$$
\phi (v,w)
:=
\mathcal{L}_{0} ( \gamma_{v} ) - \mathcal{L}_{0} ( \gamma_{w} )
$$
where for each $v \in T_{\M}M$, the curve $\gamma_{v}:[\T ,\TT ] \to M$ is given by
$$
\gamma_{v}(t) := \mathcal{L}_{0} \exp_{\M}^{\T , t} (v),
\quad \T \leq t \leq \TT .
$$
Then by the first variational formula for $\mathcal{L}_{0}$
(Proposition \ref{1stVar}), we have
\begin{align*}
( \D \phi )_{(v,w)}
=
\left(\begin{array}{cc}
\displaystyle
\big\langle
	\dot{\gamma}_{v}(\TT) ,
	( \D \mathcal{L}_{0} \exp_{\M}^{\T ,\TT} )_{v} ( \cdot )
\big\rangle_{g(\TT)},
&
\displaystyle -
\big\langle
	\dot{\gamma}_{w}(\TT),
	( \D \mathcal{L}_{0} \exp_{\M}^{\T ,\TT} )_{w} ( \cdot )
\big\rangle_{g(\TT)}
\end{array}\right) .
\end{align*}
Therefore by Remark \ref{WelDef}(ii),
the implicit function theorem tells us that
$$
N :=
\Big\{
(v,w) \in T_{\M} M \times T_{\M} M :
\begin{array}{l}
\text{$\phi (v,w) = 0$, $v \neq w$;} \\
\text{both $v$ and $w$ are regular points of $\mathcal{L}_{0} \exp_{\M}^{\T ,\TT}$}
\end{array}
\Big\}
$$
is a $(2d-1)$-dimensional submanifold of $T_{\M} M \times T_{\M} M$.

Take a countable cover $\{ (U_{i},\psi_{i}) \}_{i}$ of $M$ which consists
of local coordinate neighborhoods and consider the map
$$
\xi_{i} : N_{i} := N \cap
\big[ \big( \mathcal{L}_{0}\exp_{\M}^{\T ,\TT} \big)^{-1} (U_{i}) \big]^{2} \to \mathbb{R}^{d}
$$
defined by
$$
\xi_{i} (v,w)
:=
\psi_{i} \big( \mathcal{L}_{0} \exp_{\M}^{\T ,\TT} (v) \big)
-
\psi_{i} \big( \mathcal{L}_{0} \exp_{\M}^{\T ,\TT} (w) \big) .
$$
Under the identification $T_{\MM} M \cong T_{\psi_{i}(\MM)} \mathbb{R}^{d}$
for $\MM \in U_{i}$, we have
\begin{align*}
(\D \xi_{i})_{(v,w)}
=
\left(\begin{array}{cc}
( \D \mathcal{L}_{0} \exp_{\M}^{\T ,\TT} )_{v},
&
- ( \D \mathcal{L}_{0} \exp_{\M}^{\T ,\TT} )_{w}
\end{array}\right) .
\end{align*}
Again by the implicit function theorem, we see that
$$
N_{i}^{\prime}
:=
\Big\{
(v,w) \in N_{i} :
\begin{array}{l}
\text{%
	both 
	$( \D \mathcal{L}_{0} \exp_{\M}^{\T ,\TT} )_{v}$
	and 
	$( \D \mathcal{L}_{0} \exp_{\M}^{\T ,\TT} )_{w}$
	are%
} \\
\text{%
  non-singular and 
  $\mathcal{L}_{0} \exp_{\M}^{\T ,\TT} (v) = \mathcal{L}_{0} \exp_{\M}^{\T ,\TT} (w)$%
}
\end{array}
\Big\}
$$
is a $(d-1)$-dimensional submanifold in $T_{\M}M \times T_{\M}M$.

Now, letting
$$
\eta_{i}:N_{i}^{\prime} \ni (v,w) \mapsto \mathcal{L}_{0} \exp_{\M}^{\T ,\TT} (v) \in M,
$$
$\mathcal{C}^{\prime}$ is included
in the countable union of hypersurfaces
$\eta_{i}(N_{i}^{\prime})$ of $M$.
Therefore it has measure zero.
\end{proof}

\begin{proof}[{\bf Proof of Proposition \ref{CutLoc} (iii)}]
Assume that $\mathcal{L}_{0}\mathrm{cut}$ is not closed. Then we can take
a sequence $( \T_{i}, \M_{i}; \TT_{i}, \MM_{i} ) \in \mathcal{L}_{0}\mathrm{cut}$,
$i=1,2,\ldots$ which converges to some point
$( \T_{0}, \M_{0}; \TT_{0}, \MM_{0} ) \notin \mathcal{L}_{0}\mathrm{cut}$.
We denote
$
\Delta_{0}
:=
\{
	( t^{\prime}, t^{\prime\prime} )
	\in [0,T]^{2} :
	t^{\prime} < t^{\prime\prime}
\}
$.

Then the map
$$
\varphi :
\Delta_{0} \times TM
\to 
( [0,T] \times M )^2
$$
defined by
$$
\varphi
\big(
	\T , \TT , ( \M , v^{\prime} )
\big)
:=
\big(
	\T , \M ; \TT , \mathcal{L}_{0}\exp_{\M}^{\T ,\TT} ( v )
\big)
$$
is non-singular at $\big( \T_{0} , \TT_{0} , ( \M_{0} , v_{0}^{\prime} ) \big)$
where $v_{0}^{\prime} := \dot{\gamma} ( \TT_0 )$
and $\gamma$ is the unique $\mathcal{L}_{0}$-minimizing curve joining
$
(t_{0}^{\prime}, m_{0}^{\prime})
$
and
$
(t_{0}^{\prime\prime}, m_{0}^{\prime\prime})
$.
Hence by the inverse function theorem, we can take an open neighborhood $U$
of $\big( \T_{0} , \TT_{0} , ( \M_{0} , v_{0}^{\prime} ) \big)$ and
an open neighborhood $V$ of $( \T_{0}, \M_{0}; \TT_{0}, \MM_{0} )$ such that
$\varphi\vert_{U}:U\to V$ is diffeomorphic. Define $v_{i}^{\prime} \in T_{\M_{i}}M$
by
$
\big( \T_{i} , \TT_{i} , ( \M_{i} , v_{i}^{\prime} ) \big)
=
\varphi|_{U}^{-1} ( \T_{i}, \M_{i}; \TT_{i}, \MM_{i} ) .
$
We may assume that $( \T_{i}, \M_{i}; \TT_{i}, \MM_{i} ) \in V$ for all $i$
with neglecting finite numbers of points if necessary.

Note that, for each
$\big( \T , \TT , ( \M , v^{\prime} ) \big) \in U$,
$v^{\prime}$ is not a critical point
of $\mathcal{L}_{0}\exp_{\M}^{\T ,\TT}:T_{\M}M \to M$.
Therefore we can choose $w_{i}^{\prime} \in T_{\M_{i}}M$ so that
$[\T_{i} ,\TT_{i}] \ni t \mapsto  \mathcal{L}_{0}\exp_{\M_{0}}^{\T,t} (w_{i}^{\prime})$
is a $\mathcal{L}_{0}$-minimizing geodesic joining $(\T_{i}, \M_{i})$ to $(\TT_{i}, \MM_{i} )$
but $v_{i}^{\prime} \neq w_{i}^{\prime}$.
Taking a subsequence if necessary, we may assume that
$
\big( \T_{i} , \TT_{i} , ( \M_{i} , w_{i}^{\prime} ) \big)
\stackrel{i\to\infty}{\rightarrow}
\big( \T_{0} , \TT_{0} , ( \M_{0} , w_{0}^{\prime} ) \big)
$
for some $w_{0}^{\prime} \in T_{\M_{0}}M$. 
Since $\big( \T_{i} , \TT_{i} , ( \M_{i} , w_{i}^{\prime} ) \big) \notin U$ for all $i$
and $\varphi\vert_{U}:U\to V$ is diffeomorphic, we must have
$w_{0}^{\prime} \neq v_{0}^{\prime}$.

Consequently, the curves
$$
t \mapsto \mathcal{L}_{0}\exp_{\M_{0}}^{\T_{0} ,t} (v_{0}^{\prime})
\quad
\text{and}
\quad
t \mapsto \mathcal{L}_{0}\exp_{\M_{0}}^{\T_{0} ,t} (w_{0}^{\prime})
$$
must be distinct $\mathcal{L}_{0}$-minimizing curves joining $(\T_{0}, \M_{0})$
and $(\TT_{0}, \MM_{0})$,
which contradicts to
$(\T_{0}, \M_{0}; \TT_{0} ,\MM_{0} ) \notin \mathcal{L}_{0}\mathrm{cut}$.
Hence $\mathcal{L}_{0}\mathrm{cut}$ is closed.
\end{proof}

We introduce the notion of $\mathcal{L}_0$-Jacobi fields. 
For a smooth curve $\gamma : [ \T , \TT ] \to M$, 
a $C^2$-vector field $V$ along $\gamma$ and $t \in [ \T , \TT ]$, 
we define a linear form $\mathscr{J}_t (V)$ as follows: 
\begin{align} \label{eq:Jacobi}
\mathscr{J}_{t} (V)
&:=
- \nabla^{g(t)}_{\dot{\gamma} (t)} \nabla^{g(t)}_{\dot{\gamma} (t)} V (t)
+
\frac12 \Hess_{g(t)} R_{g(t)} ( V(t), \cdot )
\\ 
\nonumber
&\hspace{15mm}
- \mathcal{R}_{g(t)} ( V(t), \dot{\gamma}(t) ) \dot{\gamma} (t) 
+ 2
\big(
	\nabla_{ V(t) }^{ g(t) }
	\Ric_{g(t)}
\big)
( \dot{\gamma}(t), \cdot )
\\ \nonumber
& \hspace{15mm}
+
2 \Ric_{g(t)}
\big(
      \nabla_{ V(t) }^{ g(t) } \dot{\gamma} (t),
      \cdot
\big) .
\end{align}
We call a vector field $J$ along $\gamma$ an
\textit{$\mathcal{L}_{0}$-Jacobi field} 
if $\mathscr{J}_t (J) = 0$ for each $t \in [ \T , \TT ]$. 
Note that some computation yields 
the following relation between $\mathscr{J}_t$ 
and $\mathcal{L}_0 I_\gamma$: 
\begin{equation} \label{eq:I-J}
\mathcal{L}_0 I_\gamma ( V , W ) 
= 
\left.
\big\langle
	\nabla^{g(t)}_{\dot{\gamma} (t)} V(t),
	W(t)
\big\rangle_{ g(t) }
\right\vert_{ t=t^{\prime} }^{ t=t^{\prime\prime} }
+ \int_{\T}^{\TT} 
\langle \mathscr{J}_t (V) , W (t) \rangle_{g(t)} \D t . 
\end{equation}

\begin{Lem} \label{lem:J} 
Let $\gamma : [ \T , \TT ] \to M$ be an $\mathcal{L}_0$-geodesic. 
Then $\dot{\gamma} (\T)$ is a critical point of 
$\mathcal{L}_0\exp^{\T,\TT}_{\gamma (\T)}$ 
if and only if there is an $\mathcal{L}_0$-Jacobi field $J$ 
along $\gamma$ with $J (\T) = 0$, $J (\TT) = 0$ and $J \not \equiv 0$. 
\end{Lem} 

\begin{proof} 
Let $\dot{\gamma} (\T)$ be 
a critical point of $\mathcal{L}_0\exp_{\gamma(\T)}^{\T, \TT}$. 
It means that there is a non-zero vector
$V \in T_{ \dot{\gamma} ( t^{\prime} ) } ( T_{\gamma (t^{\prime})} M )$ 
satisfying $( \D \mathcal{L}_0\exp_{\gamma(\T)}^{\T ,\TT} )_{\dot{\gamma}(\T)} (V) = 0$. 
We consider a variation 
$f(u,t):
( - \varepsilon , \varepsilon )
\times
[ t^{\prime}, t^{\prime\prime} ]
\to
M$ 
of $\gamma$ given by 
\begin{equation} \label{eq:J-var}
f(u,t)
:=
\mathcal{L}_{0}
\exp_{ \gamma ( t^{\prime} ) }^{ t^{\prime}, t }
\big(
	\dot{\gamma} ( t^{\prime} ) + u ( t - t^{\prime} ) V
\big) .
\end{equation}
Then the vector field $J$ along $\gamma$ given by 
$
J(t)
:=
\left.
	\frac{ \mathrm{d} }{ \mathrm{d} u }
\right\vert_{ u=0 }
f(u,t)
$
is an $\mathcal{L}_0$-Jacobi field. 
Indeed, we can verify it by applying 
$\nabla_{ ( \partial / \partial u ) \vert_{u=0} f(u,t)  }^{g(t)}$
to the $\mathcal{L}_{0}$-geodesic equation
$\mathscr{G}_{t}( f(u,\cdot ) )=0$ for $f(u,\cdot )$. 
Then $J (\T) =0$ by definition and $J (\TT) = 0$ by the choice of $V$. 
In addition, $\nabla^{g(\T)}_{\dot{\gamma}(\T)} J (\T) = V$ holds 
(Here we are identifying $V$ with a vector in $T_{\gamma(\T)} M$). 
Since $V \neq 0$, we have $J \not \equiv 0$. 

Conversely, suppose that there is an $\mathcal{L}_0$-Jacobi field 
$J$ along $\gamma$ with $J (\T) = 0$, $J (\TT) \neq 0$ and $J \not \equiv 0$. 
Then $V: = \nabla^{g(\T)}_{\dot{\gamma} (\T)} J \neq 0$. 
Indeed, if $V = 0$, then $J \equiv 0$ must hold 
by the uniqueness of the solution to 
the second order linear differential equation $\mathscr{J}_t (J)=0$. 
By identifying $V$ with a vector in $T_{\dot{\gamma}(\T)} ( T_{\gamma (\T)} M )$, 
we consider a variation $f$ given by \eqref{eq:J-var}. 
Then again by the uniqueness of the solution to $\mathscr{J}_t (J) = 0$, 
the variational vector field corresponding to $f$ coincides with $J$. 
Then $J( \TT ) = 0$ means 
$( \D \mathcal{L}_0 \exp_{\gamma (\T)}^{\T , \TT} )_{\dot{\gamma} (\T)} (V) = 0$ 
and hence the conclusion holds. 
\end{proof} 

\begin{Prop} 
\label{extension} 
For each
$
(
	t^{\prime}, m^{\prime};
	t^{\prime\prime}, m^{\prime\prime}
)
\in \mathcal{M}
$,
the following two conditions are equivalent:
\begin{itemize}
\item[(i)]
$
(
	t^{\prime}, m^{\prime};
	t^{\prime\prime}, m^{\prime\prime}
)
\notin \mathcal{L}_{0}\mathrm{cut}
$.

\item[(ii)] Every $\mathcal{L}_{0}$-minimizing curve $\gamma$
joining
$( t^{\prime}, m^{\prime})$
and
$( t^{\prime\prime}, m^{\prime\prime} )$
extends to an $\mathcal{L}_{0}$-minimizing curve beyond $(\TT, \MM)$, 
that is,
there exist an $\varepsilon >0$ and a curve
$\widetilde{\gamma}: [t^{\prime} , t^{\prime\prime}+\varepsilon] \to M$
such that
$
\widetilde{\gamma} |_{ [ t^{\prime} , t^{\prime\prime} ] }
=
\gamma
$
and $\widetilde{\gamma}$ is an $\mathcal{L}_{0}$-minimizing curve
joining its endpoints.
\end{itemize}
\end{Prop} 
\begin{proof} 
Assume that (i) holds.
If $\gamma$ is an $\mathcal{L}_{0}$-minimizing curve joining
$( t^{\prime}, m^{\prime})$
and
$( t^{\prime\prime}, m^{\prime\prime} )$,
then $\gamma$ is unique among such curves
and satisfies the $\mathcal{L}_{0}$-geodesic equation.
Since the map
$
T_{m^{\prime}}M \ni w
\mapsto
\mathcal{L}_{0}
\exp_{ m^{\prime} }^{ t^{\prime}, t^{\prime\prime} } (w)
\in M
$
is regular at $\dot{\gamma}(t^{\prime})$,
there exists an open neighborhood $U$ of $\dot{\gamma}(t^{\prime})$
such that the map is diffeomorphic on $U$ to its image.

We extend $\gamma$ forward in time,
to $\widetilde{\gamma} : [ t^{\prime}, t^{\prime\prime} + \varepsilon_{0} ] \to M$
by solving the same $\mathcal{L}_{0}$-geodesic equation
($\varepsilon_{0} >0$ being small enough).
Since $\mathcal{L}_{0}\mathrm{cut}$ is closed (Proposition \ref{CutLoc}(iii)),
we can assume that
$
\widetilde{\gamma} [ t^{\prime\prime}, t^{\prime\prime} + \varepsilon_{0} ]
\cap
\mathcal{L}_{0}
\mathrm{cut}_{
		\widetilde{\gamma} ( t^{\prime} )
	}^{
		t^{\prime} ,
		t^{\prime\prime} + \varepsilon_{0}
	}
=
\emptyset
$
for sufficiently small $\varepsilon_{0} >0$.
For each $\varepsilon \in (0, \varepsilon_{0} ]$,
we define $w_{\varepsilon} \in T_{ \gamma ( t^{\prime} ) } M$
to be the tangent vector
$\dot{\gamma}_{\varepsilon}(t^{\prime})$ of the unique
$\mathcal{L}_{0}$-minimizing curve $\gamma_{\varepsilon}$ joining
$
( t^{\prime}, \gamma ( t^{\prime} ) )
$
and
$
( t^{\prime\prime} + \varepsilon , \gamma ( t^{\prime\prime} + \varepsilon ) )
$.
By Proposition \ref{CutLoc}(iv) and Proposition \ref{1stDer-2}, 
$L_{0}$ is smooth around
$
(
	t^{\prime}, m^{\prime};
	t^{\prime\prime} + \varepsilon ,
	\gamma ( t^{\prime\prime} + \varepsilon )
)
$
and
$$
\vert
	w_{\varepsilon}
\vert_{ g( t^{\prime} ) }^{2}
=
2
\frac{
	\partial L_{0}^{ t^{\prime}, t^{\prime\prime} + \varepsilon }
}{
	\partial t^{\prime}
}
\big(
	m^{\prime} ,
	\gamma ( t^{\prime\prime} + \varepsilon )
\big)
+
R_{ g( t^{\prime} ) } ( m^{\prime} ),
$$
so that
$
\{
	w_{\varepsilon}
\}_{ 0 < \varepsilon \leq \varepsilon_{0} }
$
is a bounded set.
Therefore, taking a subsequence, we can assume that
$w_{\varepsilon}$ converges to a vector
$w \in T_{ \dot{\gamma} ( t^{\prime} ) }M$
as $\varepsilon \downarrow 0$.

Then we show that $w_{\varepsilon} \in U$ for a sufficiently
small $\varepsilon >0$.
Since $\gamma_{\varepsilon}(t)$ smoothly depends on $w_{\varepsilon}$
and $\mathcal{L}_0$ is lower semi-continuous 
(cf.~the proof of Proposition \ref{Cpt}),
by using Lemma \ref{Continuity},
we have
\begin{equation*}
\begin{split}
\mathcal{L}_{0} ( \gamma )
&=
L_{0}^{ t^{\prime}, t^{\prime\prime} }
\big(
	\gamma ( t^{\prime} ),
	\gamma ( t^{\prime\prime} )
\big) \\
&=
\lim_{\varepsilon \to 0}
L_{0}^{ t^{\prime}, t^{\prime\prime} }
\big(
	\gamma_{\varepsilon} ( t^{\prime} ),
	\gamma_{\varepsilon} ( t^{\prime\prime} + \varepsilon )
\big) \\
&=
\lim_{\varepsilon \to 0}
\mathcal{L}_{0} ( \gamma_{ \varepsilon } )
\geq
\mathcal{L}_{0} ( \gamma_{ w } )
\geq
\mathcal{L}_{0} ( \gamma ),
\end{split}
\end{equation*}
where
$
\gamma_{w} : [ t^{\prime}, t^{\prime\prime} ] \to M
$
is the $\mathcal{L}_{0}$-geodesic of initial conditions
$
\gamma_{w} ( t^{\prime} ) = m^{\prime}
$
and
$
\dot{\gamma}_{w} ( t^{\prime} ) = w
$.
It clearly holds that
$
\gamma_{w} ( t^{\prime\prime} ) = m^{\prime\prime}
$.
Therefore, $\gamma_{w}$ is also an $\mathcal{L}_{0}$-minimizing curve joining
$( t^{\prime}, m^{\prime} )$
and
$( t^{\prime\prime}, m^{\prime\prime} )$.
By the uniqueness, we have $\gamma = \gamma_{w}$,
so that $\dot{\gamma}( t^{\prime} ) = w$.
Thus $w_{\varepsilon} \in U$ for a sufficiently small $\varepsilon >0$.

Now, for a sufficiently small $\varepsilon > 0$,
we have $\dot{\gamma} (t^{\prime})$, $w_{\varepsilon} \in U$ and
\begin{equation*}
\begin{split}
\mathcal{L}_{0}
\exp_{
	\gamma ( t^{\prime} )
}^{
	t^{\prime}, t^{\prime\prime} + \varepsilon
}
( \dot{\gamma} ( t^{\prime} ) )
=
\widetilde{\gamma} ( t^{\prime\prime} + \varepsilon )
=
\gamma_{w} ( t^{\prime\prime} + \varepsilon )
=
\mathcal{L}_{0}
\exp_{
	\gamma ( t^{\prime} )
}^{
	t^{\prime}, t^{\prime\prime} + \varepsilon
}
( w_{\varepsilon} ),
\end{split}
\end{equation*}
which implies $\dot{\gamma}(t^{\prime}) = w_{\varepsilon}$,
so that $\widetilde{\gamma} = \gamma_{\varepsilon}$.
Hence the curve $\widetilde{\gamma}$ is uniquely $\mathcal{L}_{0}$-minimizing.

Conversely, assume that (ii) holds.
Suppose that $\gamma$ and $\eta$ are two $\mathcal{L}_{0}$-minimizing curves
joining
$( t^{\prime}, m^{\prime} )$
and
$( t^{\prime\prime}, m^{\prime\prime} )$.
By the assumption (ii), $\gamma$ has an $\mathcal{L}_{0}$-minimizing
extension
$
\widetilde{\gamma}:
[ t^{\prime} , t^{\prime\prime} + \varepsilon ]
\to
M
$.
Then the piecewise smooth curve $c$ defined by
$$
c (t)
:=
\left\{\begin{array}{ll}
\eta (t) & \text{if $t \in [ t^{\prime}, t^{\prime\prime} ]$,} \\
\widetilde{\gamma} (t) & \text{if $t \in [ t^{\prime\prime}, t^{\prime\prime} + \varepsilon ]$,}
\end{array}\right.
$$
which is the concatenation of $\eta$ and
$
\widetilde{\gamma}
\vert_{ [ t^{\prime}, t^{\prime\prime} + \varepsilon ] }
$,
must be also $\mathcal{L}_{0}$-minimizing.
A standard variational argument shows that
the curve $c(t)$ becomes $C^{1}$ at $t=t^{\prime\prime}$
and must be $\mathcal{L}_{0}$-geodesic
(so that $c$ is found to be smooth).
By the uniqueness result
(with respect to the initial conditions
at time $t^{\prime\prime}+\varepsilon $)
of the ODE theory, we must have $\widetilde{\gamma}=c$,
in particular, we have $\gamma = \eta$.
Therefore, the $\mathcal{L}_{0}$-minimizing geodesic
joining
$( t^{\prime}, m^{\prime} )$
and
$( t^{\prime\prime}, m^{\prime\prime} )$
must be unique.

Next, we show that
$
\mathcal{L}_{0}
\exp_{ m^{\prime} }^{ t^{\prime}, t^{\prime\prime} }
$
is not singular at $\dot{\gamma} ( t^{\prime} ) \in T_{m^{\prime}}M$.
Suppose on the contrary, that $\dot{\gamma} ( t^{\prime} )$
is a critical point of
$
\mathcal{L}_{0}
\exp_{ m^{\prime} }^{ t^{\prime}, t^{\prime\prime} }
$.
Then there is an $\mathcal{L}_0$-Jacobi field along $\gamma$ 
with $J (\T) = 0$, $J (\TT) = 0$ and $J \not \equiv 0$ 
by Lemma \ref{lem:J}. 
Note that $\nabla^{g(\TT)}_{\gamma (\TT)} J ( \TT ) \neq 0$ holds 
(cf.~the proof of Lemma \ref{lem:J}). 
We take an $\mathcal{L}_{0}$-minimizing extension
$
\widetilde{\gamma}:
[ t^{\prime}, t^{\prime\prime} + \varepsilon ]
\to
M
$
of $\gamma$ and extend $J$ to a piecewise smooth vector field
on $[ t^{\prime}, t^{\prime\prime} + \varepsilon ]$
(which we denote again by $J$)
by requiring that
$
J\vert_{ [ t^{\prime\prime}, t^{\prime\prime} + \varepsilon ] } \equiv 0
$.
We further take a vector field $V$ 
along $\widetilde{\gamma}$ with 
$
V( t^{\prime} ) = 0
$,
$
V( t^{\prime\prime} ) = - \nabla^{g(\TT)}_{\dot{\gamma} (\TT)} J ( t^{\prime\prime}- )
$
and
$
V( t^{\prime\prime} + \varepsilon ) = 0
$,
and then consider any {\it proper} variation
$
g(u,t):
( -\delta , \delta )
\times
[ t^{\prime}, t^{\prime\prime} ]
\to
M
$
of $\widetilde{\gamma}$, with the variational vector field
$
W = J + a V
$
($a > 0$).
Since $W$ vanishes at the endpoints of $\widetilde{\gamma}$,
such a proper variation exists.
By a piecewise use of the second variational formula (Proposition \ref{2ndVar})
together with \eqref{eq:I-J} and the symmetry of $\mathcal{L}_0 I_\gamma$,
we have 
\begin{equation*}
\begin{split}
&
( \delta_{W} \delta_{W} \mathcal{L}_{0} )
( \widetilde{\gamma} )
=
-2a
\vert
	\nabla^{g(\TT)}_{\dot{\gamma}(\TT)} J ( t^{\prime\prime}- )
\vert_{ g( t^{\prime\prime} ) }^{2}
+
a^{2}
( \delta_{V} \delta_{V} \mathcal{L}_{0} )
( \widetilde{\gamma} )
\end{split}
\end{equation*}
which is negative for sufficiently small $a>0$.
On the other hand, since $\widetilde{\gamma}$ is $\mathcal{L}_{0}$-minimizing
and $g$ is proper,
it must hold that
$$
\mathcal{L}_{0} ( \widetilde{\gamma} )
=
\mathcal{L}_{0} ( g( 0, \cdot ) )
\leq
\mathcal{L}_{0} ( g( u, \cdot ) ),
\quad\text{for each $u \in ( -\delta , \delta )$}
$$
which implies
$
( \delta_{W} \delta_{W} \mathcal{L}_{0} )
( \widetilde{\gamma} )
=
\left.
	( \mathrm{d}^{2} / \mathrm{d} u^{2} )
\right\vert_{ u=0 }
\mathcal{L}_{0} ( g( u, \cdot ) )
\geq 0
$.
This is a contradiction.
Therefore $\dot{\gamma} ( t^{\prime} )$ is not a critical point of
$
\mathcal{L}_{0}
\exp_{m^{\prime}}^{ t^{\prime}, t^{\prime\prime} }
$.

Hence we have proved (ii)$\Rightarrow$(i).
\end{proof} 

Finally we study relations 
in the time-reversal and $\mathcal{L}_0$-cut locus. 
Let $\widetilde{g} (\tau) = g ( T - \tau )$. 
Then $\widetilde{g}(\tau)$ evolves under the backward Ricci flow
$$
\frac{ \partial }{ \partial \tau } \widetilde{g} (\tau )
=
2 \Ric_{ \widetilde{g}( \tau ) }.
$$
For $\tau^{\prime} < \tau^{\prime\prime}$,
we consider the corresponding functional
$$
\widetilde{\mathcal{L}}_{0} ( c )
:=
\frac{1}{2}
\int_{ \tau^{\prime} }^{ \tau^{\prime\prime} } \hspace{-3mm}
\big\{
	\vert \dot{c} (\tau ) \vert_{ \widetilde{g} (\tau ) }^{2}
	+
	R_{ \widetilde{g} (\tau ) }
	( c (\tau ) )
\big\}
\mathrm{d} \tau ,
$$
where $c : [ \tau^{\prime}, \tau^{\prime\prime} ] \to M$
is any curve.
For $x^{\prime} \in M$ we define 
the $\widetilde{\mathcal{L}}_0$-exponential map 
$
\widetilde{\mathcal{L}}_{0}
\exp_{ x^{\prime} }^{ \tau^{\prime}, \tau^{\prime\prime} }:
T_{ x^{\prime} } M \to M
$
by
$$
\widetilde{\mathcal{L}}_{0}
\exp_{ x^{\prime} }^{ \tau^{\prime}, \tau^{\prime\prime} }
(w) := \eta ( \tau^{\prime\prime} )
$$
where $\eta$ is the solution to the $\widetilde{\mathcal{L}}_{0}$-geodesic equation
\begin{equation}
\label{B-L0geod} 
\begin{split}
\left\{\begin{array}{l}
\displaystyle
\nabla_{ \dot{\eta}(\tau ) }^{ \widetilde{g}(\tau ) } \dot{\eta}(\tau )
-
\frac{1}{2} \nabla^{ \widetilde{g}(\tau ) } R_{ \widetilde{g}( \tau ) }
+
2 \Ric_{ \widetilde{g}(\tau ) } ( \dot{ \eta } ( \tau ), \cdot ) = 0, \\
\eta ( \tau^{\prime} ) = x^{\prime},
\quad
\dot{ \eta } ( \tau^{\prime} ) = w . 
\end{array}\right.
\end{split}
\end{equation}
One can see that this is actually the Euler-Lagrange equation for $\widetilde{\mathcal{L}}_{0}$.
Take $\tau^{\prime} , \tau^{\prime\prime} , \T , \TT \in [ 0 , T ]$ with 
$\tau^{\prime} < \tau^{\prime\prime}$, $\tau^{\prime} = T-  t^{\prime\prime}$
and
$
t^{\prime\prime} - t^{\prime}
=
\tau^{\prime\prime} - \tau^{\prime}
$. 
For a curve $\eta : [ \tau^{\prime} , \tau^{\prime\prime} ] \to M$, 
we define 
$
\gamma : [ t^{\prime}, t^{\prime\prime} ] \to M
$
by
$
\gamma ( t^{\prime} + s ) = \eta ( \tau^{\prime\prime} - s )
$
for
$
0 \leq s \leq t^{\prime\prime} - t^{\prime}
=
\tau^{\prime\prime} - \tau^{\prime}
$. 
We call $\gamma$ the time-reversal of $\eta$. 
By definition, $\widetilde{\mathcal{L}}_0 (\eta) = \mathcal{L}_0 (\gamma)$ holds. 
In addition, by comparing \eqref{B-L0geod} with \eqref{L0geod}, 
we can easily show that $\gamma$ is an $\mathcal{L}_0$-geodesic 
if and only if $\eta$ is an $\widetilde{\mathcal{L}}_0$-geodesic. 

\begin{Prop} 
\label{B-Analogy} 
Let $\tau^{\prime} , \tau^{\prime\prime} , \T , \TT \in [ 0 , T ]$ with 
$\tau^{\prime} < \tau^{\prime\prime}$, $\tau^{\prime} = T-  t^{\prime\prime}$
and
$
t^{\prime\prime} - t^{\prime}
=
\tau^{\prime\prime} - \tau^{\prime}
$. 

\begin{itemize}
\item[(i)]
There is more than one $\mathcal{L}_{0}$-minimizing curve
joining $( t^{\prime}, x )$ and
$( t^{\prime\prime}, y )$
iff there is more than one $\widetilde{\mathcal{L}}_{0}$-minimizing curve
joining $( \tau^{\prime}, y )$
and
$( \tau^{\prime\prime}, x )$.

\item[(ii)] 
Let $\eta : [ \tau^\prime , \tau^{\prime\prime} ] \to M$ be 
an $\widetilde{\mathcal{L}}_0$-minimizing curve 
and $\gamma$ its time-reversal. 
Then the vector $\dot{\gamma}( t^{\prime} )$ is a critical point of 
$
\mathcal{L}_{0}
\exp_{ \gamma ( t^{\prime} ) }^{ t^{\prime}, t^{\prime\prime} }
$
iff
$\dot{\eta} ( \tau^{\prime} )$ is a critical point of
$
\widetilde{\mathcal{L}}_{0}
\exp_{ \eta ( \tau^{\prime} ) }^{ \tau^{\prime}, \tau^{\prime\prime} }
$.

\item[(iii)] Define
$
\widetilde{\mathcal{L}}_{0}
\mathrm{cut}_{x^{\prime}}^{\tau^{\prime} , \tau^{\prime\prime}}
$
as the set of all points $x^{\prime\prime} \in M$ such that
there is more than one $\widetilde{\mathcal{L}}_{0}$-minimizing curve
joining $(\tau^{\prime}, x^{\prime})$ and $( \tau^{\prime\prime}, x^{\prime\prime} )$
or there is a $w \in T_{x^{\prime}}M$ such that
$
x^{\prime\prime}
=
\widetilde{\mathcal{L}}_{0}
\exp_{x^{\prime}}^{ \tau^{\prime}, \tau^{\prime\prime} }
(w)
$
and $w$ is a critical point of
$
\widetilde{\mathcal{L}}_{0}
\exp_{x^{\prime}}^{ \tau^{\prime}, \tau^{\prime\prime} }
$.
We also define
$$
\widetilde{\mathcal{L}}_{0}
\mathrm{cut}^{ \tau^{\prime}, \tau^{\prime\prime} }
:=
\big\{
	( x^{\prime}, x^{\prime\prime} ):
	x^{\prime\prime}
	\in
	\widetilde{\mathcal{L}}_{0}
	\mathrm{cut}_{x^{\prime}}^{\tau^{\prime}, \tau^{\prime\prime}}
\big\}.
$$
Then the map
$
M \times M \ni (x,y) \mapsto (y,x) \in M \times M
$
gives an isomorphism between
$
\mathcal{L}_{0}
\mathrm{cut}^{ t^{\prime}, t^{\prime\prime} }
$
and
$
\widetilde{\mathcal{L}}_{0}
\mathrm{cut}^{ \tau^{\prime}, \tau^{\prime\prime} }
$.

\end{itemize}
\begin{itemize}
\item[(iv)]
$
( x^{\prime}, x^{\prime\prime} )
\notin
\widetilde{\mathcal{L}}_{0}
\mathrm{cut}^{ \tau^{\prime}, \tau^{\prime\prime} }
$ if and only if $\widetilde{\mathcal{L}}_{0}$-minimizing curve joining
$( \tau^{\prime}, x^{\prime})$
and
$( \tau^{\prime\prime}, x^{\prime\prime} )$
can be extended beyond $( \tau^{\prime\prime}, x^{\prime\prime} )$ 
with keeping its minimality.
\end{itemize}
\end{Prop}
\begin{proof} 
The claim (i) is obvious by remarks just before Proposition \ref{B-Analogy}. 
For (ii), we introduce the notion of $\widetilde{\mathcal{L}}_0$-Jacobi field: 
For $\eta : [ \tau^{\prime} , \tau^{\prime\prime} ] \to M$, 
$\widetilde{J}$ is an $\widetilde{\mathcal{L}}_0$-Jacobi field if 
$\widetilde{\mathscr{J}}_\tau (V) = 0$, where a linear form $\widetilde{\mathscr{J}}_\tau$ 
is defined by replacing $\gamma (t)$, $\dot{\gamma}(t)$ 
and $g(t)$ in \eqref{eq:Jacobi} 
with $\eta (\tau)$, $\dot{\eta} (\tau)$ and $\widetilde{g} (\tau)$ respectively 
and changing the sign of all terms involving the Ricci curvature. 
Then the criticality of $\widetilde{\mathcal{L}}_0$-exponential map is 
also characterized by $\widetilde{\mathcal{L}}_0$-Jacobi fields 
as in Lemma \ref{lem:J} by the same argument. 
Moreover, a vector field $\widetilde{J}$ along $\eta$ is 
an $\widetilde{\mathcal{L}}_0$-Jacobi field
if and only if a vector field $J$ along the time-reversal of $\eta$ 
given by $J (\T + s) := \widetilde{J} ( \tau^{\prime\prime} - s )$ is 
an $\mathcal{L}_0$-Jacobi field. 
Then the conclusion easily follows 
by combining these observations with Lemma \ref{lem:J}. 
The assertion (iii) follows immediately from (i) and (ii). 
The assertion (iv) can be proved in the same way 
as Proposition \ref{extension}. 
\end{proof} 

\section{Construction of a coupled Brownian motions in the absence of $\mathcal{L}_{0}$-cut locus}
\label{AbsL0Cut} 
In this section, 
we will show Theorem \ref{Thm1} under the assumption that 
$\mathcal{L}_0 \mathrm{cut}$ is empty
where $\mathcal{L}_0 \mathrm{cut}$ is defined 
in section \ref{sec:L0cut}. 

Under $\mathcal{L}_0 \mathrm{cut} = \emptyset$,
$L_0^{\T, \TT} ( \M , \MM )$ is a smooth function of 
$( \T ,\M ; \TT , \MM )$
(Theorem \ref{CutLoc}(iv)), 
and for each $\M$, $\MM \in M$ and $\T < \TT$, we can take a unique minimizer
for $\mathcal{L}_{0}$ joining $(\T ,\M )$ to $(\TT ,\MM )$.
Therefore, for each pair of points, 
the space-time parallel transport 
as a map between their tangent spaces 
is uniquely determined. 
In the following, we construct a coupled Brownian motions
by space-time parallel transport, 
introduced in Definition \ref{STP}.

Let $e_{1},e_{2},\ldots ,e_{d}\in \Real^{d}$ be the canonical basis of $\Real^{d}$.
We denote by $\pi :\mathscr{F}(M) \to M$ the frame bundle over $M$ and by
$\pi :\mathscr{O}^{g(t)}(M) \to M$ 
the orthonormal frame bundle with respect to the metric $g(t)$.
Note that $\mathscr{O}^{g(t)}(M)$ varies along our Ricci flow but not for $\mathscr{F}(M)$.
We canonically have the map defined by
$$
\mathscr{F}(M)_{m} \times \Real^{d} \ni (u,x) \mapsto u.x := u(x) \in T_{m}M,
$$
where $\mathscr{F} (M)_{m}$ stands for 
the fiber at $m$ of $\mathscr{F} (M)$. 

Recall our notations on time parameters 
$\tau^{\prime} (s)$, $\tau^{\prime\prime} (s)$
given in section \ref{sec:intro}: 
We sometimes drop $s$ and simply write $\tau^{\prime}$ or $\tau^{\prime\prime}$.
Let $(U_{s},V_{s})_{0 \leq s\leq S} = (U_{s}(u^{\prime}),V_{s}(v^{\prime\prime}))_{0\leq s\leq S}$
be the solution to the SDE
\begin{align*}
\left\{\begin{array}{l}
\displaystyle
\D U_{s}
=
\sqrt{2} H_{i}^{g(\B)} (U_{s}) \circ \D W_{s}^{i}
- \sum_{\alpha ,\beta =1}^{d} \frac{\partial g}{\partial t} (\B)
\Big( U_{s}.e_{\alpha} , U_{s}.e_{\beta} \Big) \mathcal{V}^{\alpha ,\beta} (U_{s}) \D s, \\
\displaystyle
\D V_{s}
=
\sqrt{2} H_{i}^{g(\BB)} (V_{s}) \circ \D B_{s}^{i}
- \sum_{\alpha ,\beta =1}^{d} \frac{\partial g}{\partial t} (\BB)
\Big( V_{s}.e_{\alpha} , V_{s}.e_{\beta} \Big) \mathcal{V}^{\alpha ,\beta} (V_{s}) \D s, \\
\displaystyle
\D B_{s} = V_{s}^{-1} \para_{\pi U_{s}, \pi V_{s}}^{\B ,\BB} U_{s} \D W_{s} ,
\end{array}\right.
\end{align*}
starting from
$
(U_{0} , V_{0}) = (u^{\prime}, v^{\prime\prime})
\in \mathscr{O}^{g(\T_{1})}(M)_{\M} \times \mathscr{O}^{g(\TT_{1})}(M)_{\MM}
$,
where
$(H_{i}^{g(t)})_{i=1}^{d}$ is the
{\it system of canonical horizontal vector fields} on $\mathscr{F}(M)$
associated with $g(t)$, 
and $(\mathcal{V}^{\alpha ,\beta})_{\alpha ,\beta =1}^{d}$ is the
{\it system of canonical vertical vector fields},
each of which is defined by
$$
\left.
\mathcal{V}^{\alpha ,\beta} f(u)
=
\frac{\D}{\D \varepsilon}
\right\vert_{\varepsilon =0}
\hspace{-5mm}
f( u\circ \mathrm{e}^{ \varepsilon e_{\alpha} \otimes e_{\beta} } ),
\quad
f\in C^{\infty}(\mathscr{F}(M)) .
$$
It is known that the random variable $(U_{s},V_{s})$ takes its value in
$
\mathscr{O}^{ g( \tau^{\prime}(s) ) } (M)
\times
\mathscr{O}^{ g( \tau^{\prime\prime}(s) ) } (M)
$
for every $0 \leq s \leq S$
(see \cite[Proposition 1.1]{ACT2}).
We put $(X_{s}, Y_{s}) := (\pi U_{s}, \pi V_{s})$.

The next statement is an It\^o formula for the process $(X,Y)=(X_{s},Y_{s})_{0\leq s\leq S}$.
We omit the proof because it is straightforward.
\begin{Prop}
For any smooth function $f(s,\M ,\MM )$ on $[0,\infty ) \times M \times M$, we have
\begin{align*}
& \D f(s,X_{s},Y_{s}) \\
&=
\Big\{
\frac{\partial f}{\partial s} (s,X_{s},Y_{s})
+ \sum_{i=1}^{d} \Big[ \Hess_{g(\B) \oplus g(\BB)} f(s,\cdot ,\cdot ) \Big]
\Big( U_{s}.e_{i} \oplus V_{s}.e_{i}^{*}, U_{s}.e_{i} \oplus V_{s}.e_{i}^{*} \Big)
\Big\} \D s \\
&\hspace{15mm}+
[ U_{s}.e_{i} \oplus V_{s}.e_{i}^{*} ] f(s,\cdot ,\cdot ) \bullet \D W_{s}^{i} ,
\end{align*}
where $e_{i}^{*} := V_{s}^{-1} \para_{X_{s},Y_{s}}^{\B ,\BB} U_{s} .e_{i} \in \Real^{d}$
for each $i=1,2,\ldots ,d$ and $\bullet \D W_{s}^{i}$ means the It\^o integral.
\end{Prop}

\begin{Cor}
Suppose $\mathcal{L}_0 \mathrm{cut} = \emptyset$. Then we have the following: 
\begin{align*}
& \D L_{0}^{\B(s) , \BB(s)}(X_{s},Y_{s}) \\
&=
- 
\Big\{
\frac{\partial L_{0}^{t' ,t''}}{\partial t'}
+ \frac{\partial L_{0}^{t' , t''}}{\partial t''}
\Big\} 
\Big|_{(t', t'') = ( \B(s) , \BB (s) )}
(X_{s},Y_{s}) \D s \\
&\hspace{10mm}
+ \sum_{i=1}^{d} \Big[ \Hess_{g(\B) \oplus g(\BB)} L_{0}^{\B ,\BB} \Big]
\Big( U_{s}.e_{i} \oplus V_{s}.e_{i}^{*}, U_{s}.e_{i} \oplus V_{s}.e_{i}^{*} \Big)
\D s \\
&\hspace{5mm}+
[ U_{s}.e_{i} \oplus V_{s}.e_{i}^{*} ] L_{0}^{\B ,\BB} \bullet \D W_{s}^{i} .
\end{align*}
\end{Cor}

Since $\mathcal{L}_{0}\mathrm{cut} = \emptyset$,
by using the result in Proposition \ref{NonPos},
the stochastic process
$L_{0}^{ \tau^{\prime}(s), \tau^{\prime\prime}(s) } (X_{s},Y_{s})$
is a semi-martingale whose bounded variation part is non-positive.
Therefore we can conclude that
$s \mapsto L_{0}^{ \tau^{\prime}(s), \tau^{\prime\prime}(s) } (X_{s},Y_{s})$
is a supermartingale after we prove the integrability. 
For proving the integrability, we consider a family of functions 
$\varphi_n \in C^2 (\mathbb{R})$ with $\varphi_n (x) \uparrow x$ 
as $n \to \infty$ for each $x \in \mathbb{R}$. 
Suppose in addition that $\varphi_n$ is nondecreasing, 
concave and bounded from above. Since $L_0$ is bounded from below, 
we can easily show that 
$\varphi_n ( L_{0}^{ \tau^{\prime} (s), \tau^{\prime\prime} (s)} (X_{s},Y_{s}) )$ 
is a supermartingale. 
Then the integrability will be ensured by the monotone convergence theorem 
and the fact $L_0^{\tau^{\prime} (0) , \tau^{\prime\prime} (0)} ( X_0 , Y_0 )$ is deterministic.  
The proof of the integrability in the next section will go along the same idea 
but we will show it together with the rigorous proof of Theorem \ref{Thm1}.

\section{Supermartingale property of a coupled Brownian motion
in the presence of $\mathcal{L}_{0}$-cut locus}
\label{PreL0Cut} 

\subsection{Coupling via approximation by geodesic random walks}
For the construction of a suitable coupling of Brownian motions in the presence of
$\mathcal{L}_{0}$-cut locus, we use the approximation by geodesic random walks
to avoid the technical difficulty coming from  the singularity of $L_{0}$.
Indeed $L_0$ is smooth outside $\mathcal{L}_0\mathrm{cut}$ (Proposition \ref{CutLoc}(iv))
but not on $\mathcal{L}_0\mathrm{cut}$. 
To carry out this procedure, we will rely on some basic properties 
of $\mathcal{L}_0\mathrm{cut}$ summarized in section \ref{sec:L0cut}. 

We fix a measurable section (or selection)
$$
\gamma : (\T ,\M ; \TT ,\MM ) \mapsto \gamma_{\M ,\MM}^{\T ,\TT}
$$
of minimal $\mathcal{L}_{0}$-geodesics, where, the measurability is with respect to
the Borel $\sigma$-field generated by the uniform topology on the path space.
Since $\mathcal{L}_{0}$-minimizing curves with fixed endpoints
form a compact set
(Proposition \ref{Cpt}),
the existence of such a section is ensured by a measurable selection theorem 
(e.g.~\cite[Theorem 6.9.6]{B} and Proposition \ref{Cpt}).
We further fix a measurable section $\sigma (t, \cdot ) :M \to \mathscr{O}^{g(t)}(M)$
of $g(t)$-orthonormal frame bundle for each $t\in [\T_{0} ,\T_{1}]$.

To construct a geodesic random walk, we prepare an $\Real^{d}$-valued
i.i.d. sequence $(\lambda_{n})_{n=1}^{\infty}$, each of which is uniformly distributed
on the unit ball in $\Real^{d}$ centered at the origin. Since it will be needed
when working with these $\lambda$'s, we shall summarize necessary formulae
as follows. 
We omit the proof because it is obvious or 
an easy consequence of the divergence theorem. 

\begin{Lem} \label{UnifExp}
Let $V$ be an $n$-dimensional real Euclidean space.
Let $\ell :V\to \Real$ be a linear function and
$\alpha :V\times V \to \Real$ be a symmetric $2$-form on $V$.
Let $B^{n}$ is the unit ball in $V$ centered at origin. Then
\begin{itemize}
\item[(i)] $\displaystyle
\int_{B^{n}} \hspace{-3mm} \ell (x) \D x = 0
$
and
$\displaystyle
\int_{B^{n}} \hspace{-3mm} \alpha (x,x) \D x = \frac{\vol (B^{n})}{n+2} \mathrm{tr}\, \alpha ,
$ \\
where in the last equality, we have naturally regarded $\alpha$ as the linear homomorphism
$V\to V^{*} \cong V$.
\end{itemize}
Suppose further that we are given
another $n$-dimensional real vector space $W$,
a linear function $f: V\oplus W \to \Real$,
a symmetric $2$-form $\beta$ on $V\oplus W$
and a linear homomorphism $A:V\to W$. 
Then 
\begin{itemize}
\item[(ii)] $\displaystyle
\int_{B^{n}} \hspace{-3mm} f(x\oplus Ax) \D x = 0
$
and

\vspace{2mm}
\item[(iii)] $\displaystyle
\int_{B^{n}} \hspace{-3mm} \beta
\big( x \oplus A x, x \oplus A x \big) \D x
=
\frac{ \mathrm{vol} (B^{n}) }{ n+2 } \sum_{i=1}^{n} \beta
\big( e_{i} \oplus A e_{i}, e_{i} \oplus A e_{i} \big)
$ \\
where $(e_{i})_{i=1}^{n}$ is an any orthonormal basis of $V$.
\end{itemize}
\end{Lem}

Now for each $\M$, $\MM \in M$ and $\varepsilon >0$, 
let us construct a coupling 
$(X_{s}^{\varepsilon} , Y_{s}^{\varepsilon})_{0\leq s \leq S}$
of geodesic random walks starting from $(\M ,\MM)$ by
\begin{align*}
&
\left\{\begin{array}{l}
\displaystyle
X_{s}^{\varepsilon} := \exp_{X_{0}^{\varepsilon}}^{g(\B(0))}
\Big(
\frac{\sqrt{2}s}{\varepsilon} \sigma ( \B(0), X_{0}^{\varepsilon}).
\sqrt{d+2} \lambda_{1}
\Big) , 
\vspace{1ex} \\
\displaystyle
Y_{s}^{\varepsilon} := \exp_{Y_{0}^{\varepsilon}}^{g(\BB(0))}
\Big(
\frac{\sqrt{2}s}{\varepsilon} \para_{ X_{0}^{\varepsilon}, Y_{0}^{\varepsilon} }^{ \B(0),\BB(0) }
\sigma ( \B(0) , X_{0}^{\varepsilon} ).
\sqrt{d+2} \lambda_{1}
\Big) ,
\end{array}\right.
\text{for $0\leq s \leq s_{1}$,} \\
&\hspace{14mm}\vdots \\
&
\left\{\begin{array}{l}
\displaystyle
X_{s}^{\varepsilon} := \exp_{X_{s_{n}}^{\varepsilon}}^{g(\B(s_{n}))}
\Big(
\frac{\sqrt{2}s - \sqrt{2}s_{n}}{\varepsilon}
\sigma ( \B(s_{n}) , X_{s_{n}}^{\varepsilon} ).
\sqrt{d+2} \lambda_{n+1}
\Big), 
\vspace{1ex} \\
\displaystyle
Y_{s}^{\varepsilon} := \exp_{Y_{s_{n}}^{\varepsilon}}^{g(\BB(s_{n}))}
\Big(
\frac{ \sqrt{2}s - \sqrt{2}s_{n} }{\varepsilon} \para_{ X_{s_{n}}^{\varepsilon}, Y_{s_{n}}^{\varepsilon} }^{ \B(s_{n}),\BB(s_{n}) }
\sigma ( \B(s_{n}) , X_{s_{n}}^{\varepsilon} ).
\sqrt{d+2} \lambda_{n+1}
\Big)
\end{array}\right. \\
&\hspace{100mm}\text{for $s_{n} \leq s \leq s_{n+1}$} \\
&\hspace{14mm}\vdots
\end{align*}
where $X_{0}^{\varepsilon} := \M$, $Y_{0}^{\varepsilon} :=\MM$ and
$s_{n} := ( n\varepsilon^{2} ) \wedge S$ for each $n=0,1,2,\ldots$.
From Lemma \ref{UnifExp}, we see that the factor $\sqrt{d+2}$ is the normalization
constant in the sense of that
$
(d+2) \mathbf{E} \big[ \lambda_{n} \otimes \lambda_{n} \big]
=
\mathrm{id}_{\Real^{d}}.
$

We shall give a remark here. From the definition, the random curves 
$s\mapsto X_{s}^{\varepsilon}$ and $s\mapsto Y_{s}^{\varepsilon}$
are clearly piecewise smooth and then we see that
$X_{s_{n}}^{\varepsilon}$ and $Y_{s_{n}}^{\varepsilon}$ are
$\sigma \big( \lambda_{k}: 1\leq k\leq n  \big)$-measurable
and $\dot{X}_{s_{n}+}^{\varepsilon}$ and $\dot{Y}_{s_{n}+}^{\varepsilon}$ are
$\sigma \big( \lambda_{k}: 1\leq k\leq n+1  \big)$-measurable but not
$\sigma \big( \lambda_{k}: 1\leq k\leq n  \big)$-measurable.

As shown in \cite{Ku2}, each of $X^{\varepsilon} = (X_{s}^{\varepsilon})_{0\leq s\leq S}$
and $Y^{\varepsilon} = (Y_{s}^{\varepsilon})_{0\leq s\leq S}$ converges in law
to $g(\B(s))$-Brownian motion starting from $\M$
and
$g(\BB(s))$-Brownian motion starting from $\MM$ 
respectively. As a result, the collection
$\{ (X^{\varepsilon}, Y^{\varepsilon}) \}_{\varepsilon >0}$
of couplings forms a tight family and hence we can find a convergent subsequence
of $(X^{\varepsilon}, Y^{\varepsilon})_{\varepsilon >0}$.
We fix such a subsequence and denote the subsequence by the same notation
$(X^{\varepsilon}, Y^{\varepsilon} )_{\varepsilon >0}$
for simplicity. 
We denote the limit by $( X, Y ) = ( X_s , Y_s )_{0 \le s \le S}$.

\subsection{Supermartingale property}

Let $\mathcal{G}_{0}$ be
the trivial $\sigma$-field and
$\mathcal{G}_{n} := \sigma \big( \lambda_{k}: 1\leq k \leq n \big)$
for each $n=1,2,\cdots$.

\begin{Prop} \label{DiffIneq}
Set
$
\Lambda_{s}^{\varepsilon}
:= L_{0}^{ \B(s) , \BB(s) }( X_{s}^{\varepsilon}, Y_{s}^{\varepsilon} )
$.
For each relatively compact open set $M_{0}$ in $M$ including 
$X_{0}^{\varepsilon} = \M$ and $Y_{0}^{\varepsilon} =\MM$,
there is $\varepsilon_0 > 0$ which depends only on $M_{0}$ such that 
the following holds: 
For each $0 < \varepsilon \le \varepsilon_0$, 
there are 
a family $(Q_{n}^{\varepsilon})_{n=1}^{\infty}$ of random variables
and a deterministic constant
$\delta (\varepsilon)$ such that 
\begin{equation}
\label{VI} 
\Lambda_{s_{n+1}}^{\varepsilon} \leq \Lambda_{s_{n}}^{\varepsilon}
+ \varepsilon \zeta_{n+1}^{\varepsilon}
+ \varepsilon^2 \Sigma_{n+1}^{\varepsilon}
+ Q_{n+1}^{\varepsilon}
\end{equation}
with the estimate
\begin{equation}
\label{RT} 
\sum_{n:\ s_{n} < \sigma _{M_{0}} ( X^{\varepsilon}, Y^{\varepsilon} ) \wedge S}
Q_{n}^{\varepsilon} \leq \delta (\varepsilon) \to 0
\quad\text{as $\varepsilon \to 0$.}
\end{equation}
where
$\zeta_{n}^{\varepsilon}$ and $\Sigma_{n}^{\varepsilon}$
are $\mathcal{G}_{n}$-measurable and integrable random variables
such that
\begin{equation} \label{eq:E_var}
\mathbf{E}
[
	\hspace{0.5mm}
	\zeta_{n}^{\varepsilon}
	\hspace{0.5mm}
	\vert
	\hspace{0.5mm}
	\mathcal{G}_{n-1}
	\hspace{0.5mm}
] = 0
\quad
\text{and}
\quad
\overline{\Sigma}_{n}^{\varepsilon}
:=
\mathbf{E}
[
	\hspace{0.5mm}
	\Sigma_{n}^{\varepsilon}
	\hspace{0.5mm}
	\vert
	\hspace{0.5mm}
	\mathcal{G}_{n-1}
	\hspace{0.5mm}
] \leq 0
\end{equation}
for each $n \geq 1$, and
$$
\sigma_{M_{0}} ( w^{\prime}, w^{\prime\prime})
:=
\inf\big\{
s>0 : ( w^{\prime}(s) , w^{\prime\prime} (s) ) \notin M_{0} \times M_{0}
\big\}
$$
for each
$
( w^{\prime}, w^{\prime\prime})
\in C( [0,S] \to M \times M )
$.
\end{Prop} 

\begin{Rm} 
Intuitively, it is clear that the difference inequality (\ref{VI})
comes from the Taylor expansion with respect to $\varepsilon$.
Therefore, when 
$
(
	\tau^{\prime}(s_{n}), X_{s_{n}}^{\varepsilon} ;
	\tau^{\prime\prime}(s_{n}), Y_{s_{n}}^{\varepsilon}
)
\notin
\mathcal{L}_{0}\mathrm{cut}
$, 
\begin{align*}
\zeta_{n+1}^\varepsilon 
& = 
\big[ 
	\varepsilon \dot{X}_{s_{n}+}^{\varepsilon} 
	\oplus 
	\varepsilon \dot{Y}_{s_{n}+}^{\varepsilon} 
\big]
L_{0}^{ \B(s_{n}), \BB(s_{n}) }
\\
\Sigma_{n+1}^\varepsilon 
& = 
-
\Big\{
	\frac{ \partial L_{0}^{t' , t''} }{ \partial t'}
	+
	\frac{ \partial L_{0}^{t' ,t''} }{ \partial t'' }
\Big\}
\Big\vert_{(t' , t'') = ( \B(s_n) , \BB(s_n) )}
( X_{s_{n}}^{\varepsilon}, Y_{s_{n}}^{\varepsilon} ) \\
&\hspace{5mm}
+ \frac{1}{2}
\big[
	\Hess_{ g(\B(s_{n})) \oplus g(\BB(s_{n})) }
	L_{0}^{ \B(s_{n}), \BB(s_{n}) }
\big]
\big(
	\varepsilon \dot{X}_{s_{n}+}^{\varepsilon} 
        \oplus 
        \varepsilon \dot{Y}_{s_{n}+}^{\varepsilon},
	\varepsilon \dot{X}_{s_{n}+}^{\varepsilon} 
        \oplus 
        \varepsilon \dot{Y}_{s_{n}+}^{\varepsilon}
\big) . 
\end{align*}
However, we must avoid using this expression
since $\mathcal{L}_0\mathrm{cut} \neq \emptyset$. 
\end{Rm} 

\begin{proof} 
If $\mathcal{L}_{0}\mathrm{cut} = \emptyset$, 
the Taylor expansion with respect to the parameter $\varepsilon$ 
easily yields the desired estimate. 
Thus we will modify the argument by 
taking the presence of $\mathcal{L}_{0}\mathrm{cut}$ 
into account. 
We begin with determining $Q_{n+1}^{\varepsilon}$ 
which enjoys \eqref{VI} 
and we verify \eqref{RT} after that. 
Let us define $K \subset ([ 0 , S ] \times M )^3$ by 
\[
K : = 
\left\{ 
    ( u_i , m_i )_{i=1}^3 
    \; \left| \; 
    \begin{array}{l}
        ( u_1 , m_1 ), ( u_3 , m_3 ) \in [ 0 , S ] \times \overline{M_{0}}, 
        \\
        u_3 - u_1 \ge \TT_1 - \T_1 , 
        \\
        u_2 = (u_1 + u_3 )/2,
        \\
        L_0^{u_1, u_3} ( m_1 , m_3 )  
        \\ \quad 
        = L_{0}^{u_1 , u_2} ( m_1 , m_2 ) 
        + L_{0}^{u_2 , u_3} ( m_2 , m_3 ) 
    \end{array}
    \right.
\right\}.
\]
By Proposition \ref{Est0} (iii), there are constants $c_1 , c_2 > 0$ 
such that, for each $(u_{i}, m_{i})_{i=1}^{3} \in K$, 
\begin{align*}
&
\rho_{g(\T_{1})} ( m_{1}, m_{2} ) + \rho_{g(\TT_{1})} ( m_{2}, m_{3} ) \\
&\leq
c_1 ( L_{0}^{u_{1},u_{2}} ( m_{1}, m_{2} ) + L_{0}^{u_{2},u_{3}} ( m_{2}, m_{3} ) )
+ c_2 \\
& =
c_1 L_{0}^{u_{1},u_{3}} ( m_{1}, m_{3} ) + c_2
\end{align*}
Since the last quantity does not depend on $m_{2}$ and 
continuously depends on $(m_1 , t_1 ; m_3 , t_3)$ which moves on a compact set, 
there is a bounded set $D$, which is possibly larger than $M_{0}$, 
such that $( u_2 , m_2 ) \in D$ holds for any $( u_i , m_i )_{i=1}^3 \in K$.
Therefore $K$ is compact.
Take $( u_i , m_i )_{i=1}^3 \in K$. 
Then $( u_2 , m_2 )$ must be on a minimal $\mathcal{L}_{0}$-geodesic 
joining $( u_1 , m_1 )$ and $( u_3 , m_3 )$. 
We denote it by $\gamma$. 
Since $u_1 \neq u_3$ by the definition of $K$, 
both $\gamma |_{[u_1, u_{2}]}$ and $\gamma |_{[u_2 , u_3]}$ are 
not constant as a space-time curve 
(It means that $( t , \gamma (t) )$ is 
not constant both on $[ u_1 , u_2 ]$ and on $[ u_2 , u_3 ]$). 
Hence we can extend $\gamma |_{[u_1, u_{2}]}$ and $\gamma |_{[u_2 , u_3]}$ 
with keeping their minimalities. 
It implies
$( u_i , m_i ; u_{i+1} , m_{i+1} ) \notin \mathcal{L}_0\mathrm{cut}$
for $i = 1 , 2$
by Proposition \ref{extension} 
and Proposition \ref{B-Analogy}(iii)(iv).
For $j=1,2$, let $p_j$ be a projection from $K$ given by 
$p_j ( ( u_i , m_i )_{i=1}^3 ) = ( u_j , m_j ; u_{j+1} , m_{j+1} )$. 
Then $( p_j (K) )_{j=1,2}$ are compact and 
away from $\mathcal{L}_0 \mathrm{cut}$. 
Let us define a compact set 
$K_1^{\varepsilon} \subset ( [ 0, S ] \times M )^2$ 
by 
\begin{align*}
A 
: & = 
\left\{
( u' , m' , v' ; u'' , m'' , v'' ) 
\; \left| \;
\begin{array}{l}
(u' , m' ; u'', m'' ) \in p_1 (K) \cup p_2 (K), 
\\
v' \in T_{m'} M, v'' \in T_{m''} M,  
\\
\| v' \|_{g(u')} = \| v'' \|_{g(u'')} 
\\ \hspace{8em}
= \sqrt{2(d+2)} 
\end{array}
\right. 
\right\}, 
\\
K_1^{\varepsilon} 
: & = 
\Big\{ 
\left( 
    u' + a , 
    \exp_{m'}^{g (\B(u'))} ( \varepsilon a v' ) ;
    u'' + a , 
    \exp_{m''}^{g (\BB(u''))} ( \varepsilon a v'' ) 
\right) 
\\
& \hspace{10em}
\; \Big| \; 
( u' , m' , v' ; u'' , m'' , v'' ) 
\in A, a \in [ 0 , 1 ]
\Big\}. 
\end{align*}
Since $\mathcal{L}_0\mathrm{cut}$ is closed
by Proposition \ref{CutLoc}(iii), 
there is $\varepsilon_0 > 0$ such that 
$K_1^{\varepsilon} \cap \mathcal{L}_0\mathrm{cut} = \emptyset$ 
when $\varepsilon \le \varepsilon_0$. 
Note that the map 
\begin{equation*}
  ( \T , \M ; \TT , \MM ) 
  \mapsto 
  \left( 
    \frac{ \T + \TT }{2} , 
    \gamma_{\M,\MM}^{\T,\TT} 
    \left( \frac{\T + \TT}{2} \right)
  \right)
\end{equation*}
is measurable. 

Let $\varepsilon < \varepsilon_0$. 
For simplicity of notations, 
we denote the ``midpoint'' of 
$( \B(s_n) , X_{s_n}^{\varepsilon} )$ and 
$( \BB(s_n), Y_{s_n}^{\varepsilon} )$, 
and the associated variational vector 
by the following: 
\begin{align*}
\hat{\tau} 
: & = 
\frac{\B(s_n) + \BB(s_n)}{2},  
\\
\hat{X}
: & = 
\gamma_{X_{s_n}^{\varepsilon} , Y_{s_n}^{\varepsilon} }^{\B(s_n), \BB(s_n)} 
( \hat{\tau} ), 
\\
\hat{V} 
: & = 
\para_{X_{s_n}^{\varepsilon} , \hat{X}}^{\B(s_n) ,\hat{\tau}}
\dot{X}_{s_n +}^{\varepsilon}. 
\end{align*}
Then we can easily verify the following:  
\begin{align*}
\left( 
    \B(s_n) , X_{s_n}^{\varepsilon} ; 
    \hat{\tau} , \hat{X} ; 
    \BB (s_n) , Y_{s_n}^{\varepsilon} 
\right) 
& \in K, 
\\
\left(
    \B (s_n) , X_{s_n}^{\varepsilon} , \varepsilon \dot{X}_{s_n +}^{\varepsilon}; 
    \hat{\tau} , \hat{X} , \varepsilon \hat{V}
\right)
& \in K_1^{\varepsilon} , 
\\
\left(
    \hat{\tau} , \hat{X} , \varepsilon \hat{V}; 
    \BB (s_n) , Y_{s_n}^{\varepsilon} , \varepsilon \dot{Y}_{s_n +}^{\varepsilon} 
\right)
& \in K_1^{\varepsilon} . 
\end{align*}
Since $K_1^{\varepsilon} \cap \mathcal{L}_0\mathrm{cut} = \emptyset$, 
the Taylor expansion yields 
\begin{align} 
& \label{eq:V1} 
L_0^{\B(s_{n+1}), \hat{\tau} + \varepsilon^2} 
( X_{s_{n+1}}^{\varepsilon} , \exp_{\hat{X}}^{g(\hat{\tau})} ( \varepsilon \hat{V}) ) 
\\ \nonumber
& \hspace{1em} 
\le 
L_0^{\B (s_n) , \hat{\tau}} ( X_{s_{n}}^{\varepsilon} , \hat{X} )
+ 
\varepsilon 
\big[ 
  \varepsilon \dot{X}_{s_{n}+}^{\varepsilon} 
  \oplus 
  \varepsilon \hat{V}
\big]
L_{0}^{ \B(s_{n}), \hat{\tau} } 
\\ \nonumber
& \hspace{2em}
- \varepsilon^2 
\big[
        \frac{ \partial L_{0}^{t' , t''} }{ \partial t' } 
        +
            \frac{ \partial L_{0}^{t' , t'' } }{ \partial t'' } 
\big] 
\big|_{(t' , t'') = ( \tau' (s_n) , \hat{\tau} )}
( X_{s_{n}}^{\varepsilon}, \hat{X} )
\\ \nonumber 
& \hspace{2em} + 
\frac{\varepsilon^2}{2}
\Big[
  \Hess_{ g(\B(s_{n})) \oplus g(\hat{\tau}) } L_{0}^{ \B(s_{n}), \hat{\tau} }
\Big]
\big(
	\varepsilon \dot{X}_{s_{n}+}^{\varepsilon} 
        \oplus 
        \varepsilon \hat{V}, 
	\varepsilon \dot{X}_{s_{n}+}^{\varepsilon} 
        \oplus 
        \varepsilon \hat{V}
\big)
\\ \nonumber
& \hspace{2em}
+ o (\varepsilon^2)
\intertext{and}
\label{eq:V2}
& L_0^{\hat{\tau} + \varepsilon^2 , \BB(s_{n+1})} 
( \exp_{\hat{X}}^{g(\hat{\tau})} ( \varepsilon \hat{V}), Y_{s_{n+1}}^{\varepsilon} )
\\ \nonumber
& \hspace{1em} 
\le 
L_0^{\hat{\tau}, \BB (s_n)} ( \hat{X}, Y_{s_{n}}^{\varepsilon} )
+ 
\varepsilon
\big[ 
  \varepsilon \hat{V}
  \oplus 
  \varepsilon \dot{Y}_{s_{n}+}^{\varepsilon} 
\big]
L_{0}^{ \hat{\tau}, \BB(s_{n}) } 
\\ \nonumber
& \hspace{2em}
- \varepsilon^2 
\big[
        \frac{ \partial L_{0}^{t' , t''} }{ \partial t' } 
        +
            \frac{ \partial L_{0}^{t', t''} }{ \partial t'' } 
\big] 
\big|_{( t' , t'' ) = ( \hat{\tau}, \BB (s_n) )}
( \hat{X} , Y_{s_{n}}^{\varepsilon} )
\\ \nonumber
& \hspace{2em} + 
\frac{\varepsilon^2}{2}
\Big[
  \Hess_{ g(\hat{\tau}) \oplus g(\BB(s_{n})) } L_{0}^{  \hat{\tau}, \BB(s_{n}) }
\Big]
\big(
        \varepsilon \hat{V} 
        \oplus 
	\varepsilon \dot{Y}_{s_{n}+}^{\varepsilon} ,
        \varepsilon \hat{V}
        \oplus 
	\varepsilon \dot{Y}_{s_{n}+}^{\varepsilon} 
\big)
\\ \nonumber
& \hspace{2em} 
+ o (\varepsilon^2). 
\end{align}
Note that the two remainder terms $o ( \varepsilon^2 )$ 
appeared in the last equalities consist of 
higher derivatives of $L_0$ on $K_1^{\varepsilon_0}$. 
We denote the sum of these two remainder terms by $Q_{n+1}^{\varepsilon}$. 
Since $K_1^{\varepsilon_0}$ is compact, $Q_{n+1}^{\varepsilon}$ is controlled 
uniformly in the position of $( X_{s_n}^\varepsilon , Y_{s_n}^{\varepsilon} )$ 
and $n$ as long as $s_n < \sigma_M ( X^\varepsilon , Y^\varepsilon )$. 
It means that there is a constant $\tilde{\delta} ( \varepsilon )$ 
being independent of $( X_{s_n}^\varepsilon , Y_{s_n}^{\varepsilon} )$ and $n$ 
such that $Q_{n+1}^{\varepsilon} \le \tilde{\delta} (\varepsilon)$ and 
$\tilde{\delta} ( \varepsilon ) / \varepsilon^2 \to 0$. 
The triangular inequality 
for $L_0$ and the definition of $( \hat{\tau}, \hat{X} )$ 
together with \eqref{eq:V1} and \eqref{eq:V2} yield 
\begin{align*}
&
L_0^{\B (s_{n+1}) , \BB ( s_{n+1} )} ( X_{s_{n+1}}^{\varepsilon} , Y_{s_{n+1}}^{\varepsilon} ) 
-
L_0^{\B (s_{n}) , \BB ( s_{n} )} ( X_{s_{n}}^{\varepsilon} , Y_{s_{n}}^{\varepsilon} ) 
\\
& \hspace{1em} 
\le 
\left( 
    L_0^{ \B (s_{n+1}) , \hat{\tau} + \varepsilon^2 } 
    ( 
        X_{s_{n+1}}^{\varepsilon} , 
        \exp_{\hat{X}}^{g(\hat{\tau})} ( \varepsilon \hat{V})
    ) 
    - 
    L_0^{ \B (s_{n}) , \hat{\tau} } 
    ( 
        X_{s_{n}}^{\varepsilon} , 
        \hat{X}
    ) 
\right)
\\
& \hspace{2em} + 
\left( 
    L_0^{\hat{\tau} + \varepsilon^2  , \BB ( s_{n+1} )} 
    ( 
        \exp_{\hat{X}}^{g(\hat{\tau})} ( \varepsilon \hat{V}) , 
        Y_{s_{n+1}}^{\varepsilon} 
    ) 
    - 
    L_0^{\hat{\tau} , \BB ( s_{n} )} 
    ( 
        \hat{X}, 
        Y_{s_{n}}^{\varepsilon}  
    ) 
\right)
\\
& \hspace{1em} = 
\varepsilon \zeta_{n+1}^{\varepsilon}
+ \varepsilon^2 \Sigma_{n+1}^{\varepsilon}
+ Q_{n+1}^{\varepsilon}, 
\end{align*}
where random variables 
$\zeta_{n+1}^{\varepsilon}$
and
$\Sigma_{n+1}^{\varepsilon}$
are defined by
\begin{equation*}
\begin{split}
\zeta_{n+1}^{\varepsilon}
:=
[ \varepsilon \dot{X}_{s_{n}+}^{\varepsilon} \oplus \varepsilon \hat{V} ]
L_{0}^{ \tau^{\prime} (s_{n}) , \hat{\tau} }
+
[ \varepsilon \hat{V} \oplus \varepsilon \dot{Y}_{s_{n}+}^{\varepsilon} ]
L_{0}^{ \hat{\tau} , \tau^{\prime\prime} (s_{n}) }
\end{split}
\end{equation*}
and
\begin{equation*}
\begin{split}
\Sigma_{n+1}^{\varepsilon}
&:=
-
\Big\{
	\frac{
		\partial L_{0}^{ t^{\prime} , t^{\prime\prime} }
	}{
		\partial t^{\prime}
	}
	+
	\frac{
		\partial L_{0}^{ t^{\prime} , t^{\prime\prime} }
	}{
		\partial t^{\prime\prime}
	}
\Big\}
\Big\vert_{
	( t^{\prime} , t^{\prime\prime} )
	=
	( \tau^{\prime} (s_{n}) , \hat{\tau} )
}
(
	X_{s_{n}}^{\varepsilon} , \hat{X}
) \\
&\hspace{10mm} +
\frac{1}{2}
\big[
	\Hess_{ g( \tau^{\prime} (s_{n}) ) \oplus g( \hat{\tau} ) }
	L_{0}^{ \tau^{\prime} (s_{n}), \hat{\tau} }
\big]
\big(
	\varepsilon \dot{X}_{s_{n}+}^{\varepsilon}
	\oplus
	\varepsilon \hat{V} ,
	\varepsilon \dot{X}_{s_{n}+}^{\varepsilon}
	\oplus
	\varepsilon \hat{V}
\big) \\
&\hspace{3mm}
-
\Big\{
	\frac{
		\partial L_{0}^{ t^{\prime} , t^{\prime\prime} }
	}{
		\partial t^{\prime}
	}
	+
	\frac{
		\partial L_{0}^{ t^{\prime} , t^{\prime\prime} }
	}{
		\partial t^{\prime\prime}
	}
\Big\}
\Big\vert_{
	( t^{\prime} , t^{\prime\prime} )
	=
	( \hat{\tau} , \tau^{\prime\prime} (s_{n}) )
}
(
	\hat{X} , Y_{s_{n}}^{\varepsilon}
) \\
&\hspace{13mm}+
\frac{1}{2}
\big[
	\Hess_{ g( \hat{\tau} ) \oplus g( \tau^{\prime\prime} (s_{n}) ) }
	L_{0}^{ \hat{\tau} , \tau^{\prime} (s_{n}) }
\big]
\big(
	\varepsilon \hat{V}
	\oplus
	\varepsilon \dot{Y}_{s_{n}+}^{\varepsilon} ,
	\varepsilon \hat{V}
	\oplus
	\varepsilon \dot{Y}_{s_{n}+}^{\varepsilon}
\big) . 
\end{split}
\end{equation*}
Thus \eqref{VI} holds. Since  
\[
\# \{ k \in \Natural : s_{k} < \sigma _{M_{0}} ( X^{\varepsilon}, Y^{\varepsilon} ) \wedge S \} 
\le S \varepsilon^{-2} , 
\]
we have 
\[
\sum_{n:\ s_{n} < \sigma _{M_{0}} ( X^{\varepsilon}, Y^{\varepsilon} ) \wedge S}
Q_{n+1}^{\varepsilon} \le S \varepsilon^{-2} \tilde{\delta} ( \varepsilon ) \to 0 
\]
as $\varepsilon \to 0$. It asserts \eqref{RT}.

Finally, we prove the required properties for
$\zeta_{n}^{\varepsilon}$ and $\Sigma_{n}^{\varepsilon}$.
The measurability are obvious. The integrabilities hold  
because $X_{s_{n}}^{\varepsilon}$, $\hat{X}$ and $Y_{s_{n}}^{\varepsilon}$
lie on a bounded domain in $M$.
Finally \eqref{eq:E_var} follows from 
Lemma \ref{UnifExp} and Proposition \ref{NonPos} 
since $\lambda_{n}$ is independent of $\mathcal{G}_{n-1}$. 
\end{proof}

Let $(\mathcal{F}_{s})_{0\leq s \leq S}$ be the filtration defined by
$$
\mathcal{F}_{s} := \sigma \big( (X_{u},Y_{u}) : 0\leq u \leq s \big) ,
\quad
0\leq s \leq S
$$
and set $( \mathcal{F}_{s_{n}}^{\varepsilon} )_{n=0}^{\infty}$ by
$
\mathcal{F}_{0}^{\varepsilon}
:= \text{the trivial $\sigma$-field,}
$
and
\begin{align*}
\mathcal{F}_{s_{n}}^{\varepsilon}
&:= \sigma \big( (X_{s_{k}}^{\varepsilon},Y_{s_{k}}^{\varepsilon}) : k=1,2,\ldots ,n \big)
\end{align*}
for each $n=1,2,\ldots$.

\begin{Thm} \label{thm:supermart} 
Set $\Lambda_{s}:=L_{0}^{ \B(s) , \BB(s) } ( X_{s},Y_{s} )$.
Then $\Lambda_s$ is integrable for each $s \in [ 0 , S ]$ and 
for each $u \le s$ we have
\[
\mathbf{E} [ \Lambda_{s} \vert \mathcal{F}_{u} ]
\leq
\Lambda_{u}, 
\]
that is, 
$\Lambda =(\Lambda_{s})_{0\leq s\leq S}$ is an
$(\mathcal{F}_{s})_{0\leq s\leq S}$-supermartingale.
\end{Thm}
\begin{proof}
Take $\eta > 0$ arbitrarily and 
choose a bounded open $M_{0} \subset M$ sufficiently large so that 
$
\mathbf{P} [ \sigma_{M_{0}} ( X,Y ) \le s ] \le \eta
$.
We first claim that 
for each $\beta > 0$ there exists $C = C ( \beta ) > 0$ 
being independent of $M_{0}$ and $\eta$ such that 
\begin{equation} \label{eq:supermart0}
\mathbf{E} \big[ 
( \beta \wedge \Lambda_{s \wedge \sigma_{M_{0}}(X,Y) } 
- 
\beta \wedge \Lambda_{u \wedge \sigma_{M_{0}}(X,Y) } ) 
F \big] 
\le C \eta 
\end{equation}
for each nonnegative bounded $\mathcal{F}_{u}$-measurable random variable $F$.
For proving \eqref{eq:supermart0}, we may assume that $F$ is of the form 
$
F=f( Z_{u_{1}}, \ldots , Z_{u_{n}} )
$
where $f$ is a nonnegative and bounded continuous function on $( M^{2} )^{n}$,
$Z_{u}=(X_{u},Y_{u})$ and $0 \leq u_{1} < \cdots < u_{n} \leq u$.

By Proposition \ref{Est0} (iii), we have 
\begin{align*}
\mathbf{E} & \big[ 
(
	\beta \wedge \Lambda_{s \wedge \sigma_{M_{0}}(X,Y) } 
	-
	\beta \wedge \Lambda_{u \wedge \sigma_{M_{0}}(X,Y) }
)
F
\big] 
\\
& \le 
\mathbf{E}
\big[ 
	( \beta \wedge \Lambda_{s} - \beta \wedge \Lambda_{u} )
	F :
	\sigma_{M_{0}} (X,Y) > s
\big] 
+ C_1 \eta 
\end{align*}
for some constant $C_1 > 0$ which is independent of $M_{0}$
but may depend on $\beta$
(Indeed, one can take $C_{1} = \beta + ( dK_{-}T ) / 2$). 
Let
$
[u]_{\varepsilon}
:=
\sup
\big\{
	\varepsilon^{2} n :
	n \in \Natural,\, \varepsilon^2 n \le u
\big\}
$ 
and $Z^\varepsilon := ( X^\varepsilon , Y^\varepsilon )$. 
Since
$
\{
	\mathbf{w}
	\; | \;
	\sigma_{M_{0}} ( \mathbf{w} ) > s
\}
$
is open, 
\begin{align*}
\mathbf{E} & \big[ 
( \beta \wedge \Lambda_{s}
- 
\beta \wedge \Lambda_{u} )
F : \sigma_{M_{0}} (X,Y) > s \big] 
\\
& \le 
\liminf_{\varepsilon \to 0} 
\mathbf{E} \big[ 
( \beta \wedge \Lambda_{s}^\varepsilon 
- 
\beta \wedge \Lambda_{u}^\varepsilon )
f ( Z^\varepsilon_{u_1} , \ldots , Z^\varepsilon_{u_n} ) 
: \sigma_{M_{0}} (X^\varepsilon,Y^\varepsilon) > s \big] 
\\
& = 
\liminf_{\varepsilon \to 0} 
\mathbf{E} \big[ 
( \beta \wedge \Lambda_{[s]_{\varepsilon}}^\varepsilon - \beta \wedge \Lambda_{[u]_\varepsilon}^\varepsilon )
f ( Z^\varepsilon_{[u_1]_\varepsilon} , \ldots , Z^\varepsilon_{[u_n]_\varepsilon} ) 
: \sigma_{M_{0}} (X^\varepsilon,Y^\varepsilon) > s \big], 
\end{align*}
where the last equality follows from the continuity of $L_0^{\B,\BB}$ and $f$. 
Then 
\begin{align*}
\mathbf{E} &
\big[ 
(
	\beta \wedge \Lambda_{[s]_{\varepsilon}}^\varepsilon
	-
	\beta \wedge \Lambda_{[u]_\varepsilon}^\varepsilon
)
f
(
	Z^\varepsilon_{[u_1]_\varepsilon} ,
	\ldots ,
	Z^\varepsilon_{[u_n]_\varepsilon}
) 
: \sigma_{M_{0}} ( X^\varepsilon , Y^\varepsilon ) > s
\big]
\\
& \le 
\mathbf{E} \big[ 
( 
	\beta
	\wedge
	\Lambda_{
		[s]_{\varepsilon}
		\wedge
		[\sigma_{M_{0}} ( X^\varepsilon , Y^\varepsilon )]_\varepsilon
	}^\varepsilon 
	- 
	\beta
	\wedge
	\Lambda_{
		[u]_{\varepsilon}
		\wedge
		[\sigma_{M_{0}} ( X^\varepsilon , Y^\varepsilon )]_\varepsilon
	}^\varepsilon 
)
f ( Z^\varepsilon_{[u_1]_\varepsilon} , \ldots , Z^\varepsilon_{[u_n]_\varepsilon} ) 
\big]
\\
& \hspace{1em}
+ C_1 \mathbf{P} [ \sigma_{M_{0}} (X^\varepsilon,Y^\varepsilon) \le s ]. 
\end{align*}
For the second term of the right hand side, 
\[
\limsup_{\varepsilon \to 0}
\mathbf{P} [ \sigma_{M_{0}} (X^\varepsilon,Y^\varepsilon) \le s ] 
\le \eta 
\]
holds since $\{ \mathbf{w} \; | \; \sigma_{M_{0}} ( \mathbf{w} ) \le s \}$ is closed.
Let us estimate the first term. 
Since Proposition \ref{DiffIneq} ensures
$\overline{\Sigma}_{i}^{\varepsilon} \leq 0$ for each $i=1,2,\ldots$, 
Proposition \ref{DiffIneq} together with the conditional Jensen inequality
yields 
\begin{align*}
&
\mathbf{E}
\big[
( 
        \beta
	\wedge
	\Lambda_{
		[s]_{\varepsilon}
		\wedge
		[\sigma_{M_{0}} ( X^{\varepsilon}, Y^{\varepsilon} )]_{\varepsilon} 
	}^\varepsilon
)
f
(
	Z_{ [u_{1}]_{\varepsilon} }^\varepsilon ,
	\ldots ,
	Z_{ [u_{n}]_{\varepsilon} }^\varepsilon
)
\big] \\
&\leq
\mathbf{E}
\Bigg[
\beta \wedge
\Bigg\{
\mathbf{E}
\Big[
	\sum_{
		i =
		[u]_{\varepsilon}
		\wedge
		[\sigma_{M_{0}} ( X^{\varepsilon}, Y^{\varepsilon} )]_{\varepsilon} + 1
	}^{
		[s]_{\varepsilon}
		\wedge
		[\sigma_{M_{0}} ( X^{\varepsilon}, Y^{\varepsilon} )]_{\varepsilon}
	}
	\hspace{-9mm} 
	\varepsilon \zeta_{i}^{\varepsilon} 
	 +
	\varepsilon^2 ( \Sigma_{i}^{\varepsilon} - \overline{\Sigma}_{i}^{\varepsilon} )
	 + 
	Q_i^\varepsilon 
	\hspace{1mm}
	\big\vert
	\mathcal{F}_{ [u]_{\varepsilon} }^{\varepsilon}
\Big]
\\
& \hspace{10mm}+
\Lambda_{
	[u]_{\varepsilon}
	\wedge
	[\sigma_{M_{0}} ( X^{\varepsilon}, Y^{\varepsilon} )]_{\varepsilon}
}^\varepsilon
\Bigg\}
f( Z_{ [u_{1}]_{\varepsilon} }^\varepsilon , \ldots , Z_{ [u_{n}]_{\varepsilon} }^\varepsilon )
\Bigg] .
\end{align*}
Since $[\sigma_{M_{0}} ( X^{\varepsilon}, Y^{\varepsilon} )]_{\varepsilon}$
is an $( \mathcal{F}_{s_{n}}^{\varepsilon} )_{n}$-stopping time,
we can apply the optional sampling theorem to conclude that
the terms involving 
$\zeta_i^\varepsilon$ and 
$\Sigma_i^\varepsilon - \bar{\Sigma}_i^\varepsilon$ vanish. 
Again by Proposition \ref{DiffIneq}, we have
\begin{align*}
\mathbf{E}
\Big[
	\sum_{
		i =
		[u]_{\varepsilon}
		\wedge
		[\sigma_{M_{0}} ( X^{\varepsilon}, Y^{\varepsilon} )]_{\varepsilon}
		+ 1
	}^{
		[s]_{\varepsilon}
		\wedge
		[\sigma_{M_{0}} ( X^{\varepsilon}, Y^{\varepsilon} )]_{\varepsilon}
	}
	\hspace{-9mm}
	Q_{i}^{\varepsilon}
	\hspace{3mm}
	\big\vert
	\mathcal{F}_{ [u]_{\varepsilon} }^{\varepsilon}
\Big]
\leq
\delta (\varepsilon) \to 0
\quad \text{as $\varepsilon \downarrow 0$.}
\end{align*}
Therefore, by combining all these estimates, 
we obtain the claim \eqref{eq:supermart0} with $C = 2 C_1$.

In \eqref{eq:supermart0}, 
letting $M_{0} \uparrow M$ 
with the dominated convergence theorem 
and 
letting $\eta \downarrow 0$ yield 
\[
\mathbf{E} \big[ 
( \beta \wedge \Lambda_{s} )
F \big] 
\le 
\mathbf{E} \big[ 
( \beta \wedge \Lambda_{u} ) 
F \big] 
\]
Since $\Lambda_s$ is bounded from below by Proposition \ref{Est0} (iii), 
the monotone convergence theorem yields the conclusion
by $\beta \uparrow \infty$. 
Indeed, we obtain the integrability of $\Lambda_s$ 
by applying this argument with $u=0$. 
\end{proof}

Now we are in turn to complete the proof of Theorem \ref{Thm1}. 

\begin{proof}[Proof of Theorem \ref{Thm1}]
In Theorem \ref{thm:supermart},
we have proved the existence of a coupling $(X,Y)$
of $g(\tau^{\prime}(s))$- and $g(\tau^{\prime\prime}(s))$-Brownian motions 
with deterministic initial data such that
$
s \mapsto
L_{0}^{ \tau^{\prime}(s), \tau^{\prime\prime}(s) }
( X_{s}, Y_{s} )
$
is a supermartingale.
Thus it suffices to show that 
we can choose the family of laws of couplings 
as a measurable function of initial data.

To complete it, we 
shall employ a measurable selection theorem.
Note that the space of 
all Borel probability measures on the path space 
$C( [0,S] \to M \times M)$ equipped with the weak topology
is a Polish space.

We define $\mathscr{K}$ as the set of all laws of a coupling $(X,Y)$ of
$g(\tau^{\prime}(s))$-Brownian motion $X=(X_{s})_{0 \leq s \leq S}$
and
$g(\tau^{\prime\prime}(s))$-Brownian motion $Y=(Y_{s})_{0 \leq s \leq S}$
such that
\begin{itemize}
\item[(a)] $(X_{0},Y_{0})$ is deterministic and

\vspace{2mm}
\item[(b)] For $0 \leq u \leq s$, it holds that
$L_0^{\tau^\prime (s) , \tau^{\prime\prime} (s) } ( X_s , Y_s )$ is integrable 
and 
$$
\mathbf{E}
[
L_{0}^{ \tau^{\prime} (s), \tau^{\prime\prime} (s) }
(X_{s}, Y_{s})
\vert
\mathcal{F}_u
]
\leq
L_{0}^{ \tau^{\prime} (u), \tau^{\prime\prime} (u) }
(X_{u}, Y_{u})
\quad\text{a.s.}
$$
\end{itemize}

We denote the probability or the expectation 
with respect to $Q \in \mathcal{K}$ 
by $\mathbf{P}^{Q}$ or $\mathbf{E}^{Q}$ respectively. 
For each $\M , \MM \in M$, let us define 
$\mathcal{K}_{\M , \MM} \subset \mathcal{K}$ 
by 
\[
\mathcal{K}_{\M , \MM} 
: =
  \{ 
    Q \in \mathcal{K} 
    \; | \; 
    \mathbf{P}^Q [ ( X_0 , Y_0 ) = ( \M , \MM ) ] = 1 
  \}. 
\]
By \cite[Theorem 6.9.6]{B}, the claim holds once we show 
that $\mathcal{K}_{\M, \MM}$ is compact for each $\M , \MM$. 
Since the marginal distributions of elements in $\mathcal{K}_{\M, \MM}$
is fixed, the Prokhorov theorem yields that $\mathcal{K}_{\M, \MM}$ is 
relatively compact. 
To show that $\mathcal{K}_{\M , \MM}$ is closed, 
take a sequence $Q_n \in \mathcal{K}_{\M , \MM}$ 
which converges to $Q$. 
The following argument is similar to 
the one in the proof of Theorem \ref{thm:supermart}. 
First, $\mathbf{P}^Q [ ( X_0 , Y_0 ) = ( \M , \MM ) ] = 1$ obviously holds 
and hence $Q$ verifies the condition (a). 
Second, for each  $k \in \mathbb{N}$, $f_i \in C_b ( M^2 )$ ($i=1, \ldots k$), 
$0 \le u_1 \le u_2 \le \cdots \le u_k \le s \le S$ and $R > 0$, 
\begin{multline*}
\mathbf{E}^{Q_n} 
\left[ 
    L_0^{\tau^{\prime} (s) , \tau^{\prime\prime} (s)} ( X_s , Y_s ) \wedge R 
    \; \left| \; 
    \prod_{i=1}^k f_i ( X_{u_i} , Y_{u_i} )
    \right.
\right] 
\\
\leq
\mathbf{E}^{Q_n} 
\left[ 
    L_0^{\tau^{\prime} (u_k) , \tau^{\prime\prime} (u_k) } ( X_{u_k} , Y_{u_k} ) \wedge R 
    \; \left| \; 
    \prod_{i=1}^k f_i ( X_{u_i} , Y_{u_i} )
    \right.
\right]. 
\end{multline*}
Thus, by tending $n \to \infty$ and $R \to \infty$ after it, 
we obtain 
\begin{multline*}
\mathbf{E}^{Q} 
\left[ 
    L_0^{\tau^{\prime} (s) , \tau^{\prime\prime} (s)} ( X_s , Y_s ) 
    \; \left| \; 
    \prod_{i=1}^k f_i ( X_{u_i} , Y_{u_i} )
    \right. 
\right] 
\\
\leq
\mathbf{E}^{Q} 
\left[ 
    L_0^{\tau^{\prime} (u_k) , \tau^{\prime\prime} (u_k)} ( X_{u_k} , Y_{u_k} ) 
    \; \left| \; 
    \prod_{i=1}^k f_i ( X_{u_i} , Y_{u_i} )
    \right. 
\right]. 
\end{multline*}
In particular, by applying the same argument for $k=1$, 
$f_1 \equiv 1$ and $u_1 = 0$, we obtain that 
$L_0^{\tau^{\prime} (s) , \tau^{\prime\prime} (s)} ( X_s , Y_s )$ 
is integrable with respect to $Q$. 
Thus $Q$ verifies the condition (b) and hence $Q \in \mathcal{K}_{\M, \MM}$. 
It means $\mathcal{K}_{\M , \MM}$ is closed and the proof is completed. 
\end{proof} 

\begin{proof}
\textit{of Corollary \ref{Cor1}.} Since 
Proposition \ref{Est0} (iii) and Lemma \ref{Continuity} ensure that
$L_0^{t_1', t_1''}$ is continuous and bounded from below, 
there is a minimizer
$\pi \in \Pi \big( c^{\prime} ( t_{1}^{\prime}), c^{\prime\prime} ( t_{1}^{\prime\prime}) \big)$
for
$$
C_{ 0, \varphi }^{ t_{1}^{\prime}, t_{1}^{\prime\prime} }
\big(
	c^{\prime} ( t_{1}^{\prime}),
	c^{\prime\prime} ( t_{1}^{\prime\prime})
\big)
=
C_{ 0, \varphi }^{ \tau^{\prime}(0), \tau^{\prime\prime}(0) }
\big(
	c^{\prime} ( \tau^{\prime}(0) ),
	c^{\prime\prime} ( \tau^{\prime\prime}(0) )
\big) .
$$
Let $\mathbf{P}_{ ( m^{\prime}, m^{\prime\prime} ) }$
be the law of the coupling $(X,Y)$ with $(X_{0},Y_{0})=( m^{\prime}, m^{\prime\prime} )$ 
obtained in Theorem~\ref{Thm1}. 
This is a probability measure on $\mathscr{W}(M) \times \mathscr{W}(M)$,
where $\mathscr{W}(M)$ is the space of continuous paths in $M$ defined on $[0,S]$.
Let $\mathbf{P}$ be the probability measure 
on $\mathscr{W}(M) \times \mathscr{W}(M)$ given by  
\[
\mathbf{P} ( \D w' , \D w'' ) 
: = 
\int_{M\times M} \hspace{-5mm} \pi ( \D m^{\prime}, \D m^{\prime\prime} )
\mathbf{P}_{ ( m^{\prime}, m^{\prime\prime} ) } ( \D w^{\prime}, \D w^{\prime\prime} ). 
\]
Note that $\mathbf{P}$ is well-defined 
by virtue of the measurability result in Theorem~\ref{Thm1}. 
Under $\mathbf{P}$, 
the canonical process $( w^{\prime}(s), w^{\prime\prime}(s) )$ is a coupling of
$g(\tau^{\prime}(s))$-Brownian motion 
and
$g(\tau^{\prime\prime}(s))$-Brownian motion 
with the initial distribution $\pi$ such that
$
L_{0}^{ \tau^{\prime}(s), \tau^{\prime\prime}(s) }
( w^{\prime}(s), w^{\prime\prime}(s) )
$
is a supermartingale. 
In particular, the law of $( w^{\prime}(s) , w^{\prime\prime}(s) )$ gives 
a coupling of
$c^{\prime}(\tau^{\prime}(s))$
and
$c^{\prime\prime}(\tau^{\prime\prime}(s))$ 
for each $s \in [ 0 , S ]$. 
Since 
$
\varphi
\big(
	L_{0}^{ \tau^{\prime}(s), \tau^{\prime\prime}(s) }
	( w^{\prime}(s), w^{\prime\prime}(s) )
\big)
$
is still a supermartingale under $\mathbf{P}$, 
we have
\begin{align*}
&
C_{ 0, \varphi }^{ \tau^{\prime}(s), \tau^{\prime\prime}(s) }
\Big(
	c^{\prime} ( \tau^{\prime}(s) ),
	c^{\prime\prime} ( \tau^{\prime\prime}(s) )
\Big) \\
&\leq
\int_{ M \times M } \hspace{-5mm}
\mathbf{E}_{ ( m^{\prime}, m^{\prime\prime} ) }
\Big[
   \varphi \left( 
	L_{0}^{ \tau^{\prime}(s), \tau^{\prime\prime}(s) }
	( w^{\prime}(s), w^{\prime\prime}(s) )
   \right)     
\Big]
\pi ( \D m^{\prime}, \D m^{\prime\prime} ) \\
&\leq
\int_{ M \times M } \hspace{-5mm}
\mathbf{E}_{ ( m^{\prime}, m^{\prime\prime} ) }
\Big[
   \varphi \left( 
	L_{0}^{ \tau^{\prime}(0), \tau^{\prime\prime}(0) }
	( w^{\prime}(0), w^{\prime\prime}(0) )
   \right)
\Big]
\pi ( \D m^{\prime}, \D m^{\prime\prime} ) \\
&=
\int_{ M \times M } \hspace{-5mm}
\varphi \left( 
    L_{0}^{ \tau^{\prime}(0), \tau^{\prime\prime}(0) }
    ( m^{\prime}, m^{\prime\prime} )
\right)
\pi ( \D m^{\prime}, \D m^{\prime\prime} ) \\
& =
C_{ 0, \varphi }^{ \tau^{\prime}(0), \tau^{\prime\prime}(0) }
\Big(
	c^{\prime} ( \tau^{\prime}(0) ),
	c^{\prime\prime} ( \tau^{\prime\prime}(0) )
\Big),
\end{align*}
where
$\mathbf{E}_{ ( m^{\prime}, m^{\prime\prime} ) }$
stands for the expectation with respect to
$\mathbf{P}_{ ( m^{\prime}, m^{\prime\prime} ) }$. 
Hence 
$
C_{ 0, \varphi }^{ t_{1}^{\prime}-s, t_{1}^{\prime\prime}-s }
\big(
	c^{\prime} ( t_{1}^{\prime} - s ),
	c^{\prime\prime} ( t_{1}^{\prime\prime} - s )
\big)
$ is non-increasing in $s$ because we can repeat the same argument 
even if we replace the initial time $0$ with any $s' \in [ 0 , s ]$. 
\end{proof}

\begin{Rm} \label{Correction} 
The proof of the integrability of the $\mathcal{L}$-distance 
between the coupling of Brownian motions by space-time parallel transport 
in \cite{KP2} seems to be incorrect. 
It is worth mentioning that we can recover the same integrability 
as in Theorem \ref{thm:supermart} even in that case 
by the same argument. 

More precisely, the argument in \cite[Lemma 6]{KP2} 
seems to require some modification. 
In the proof, they claimed an inequality 
analogous to Doob's $L^p$-martingale inequality 
for a positive supermartingale. 
However it is not true in general. 
Indeed there is a counterexample to Doob's inequality 
when $p=1$ for a positive martingale $M_n$ (See \cite[Example 5.4.2]{Du}). 
Since the $x \mapsto x^{1/p}$ is nonincreasing and concave 
on $[ 0 , \infty )$ for $p \ge 1$, 
$M_n^{1/p}$ gives a counterexample to Doob's $L^p$-inequality to 
positive supermartingales. 

On the other hand, if we further assume the stronger restriction 
on the Ricci curvature, we can recover a similar integrability 
as stated in \cite[Lemma 6]{KP2}:
\end{Rm} 

\begin{Prop} \label{Int'ble} 
Suppose $\sup_{0 \le t \le T} | \Ric |_{g(t)} < \infty$. 
Fix arbitrary two points $\M$, $\MM \in M$.
For any coupling $(X_{s},Y_{s})_{0\leq s \leq S}$ of $g(\B(s))$-Brownian motion
$X=(X_{s})_{0\leq s \leq S}$ with $X_{0}=\M$  and
$g(\BB(s))$-Brownian motion $Y=(Y_{s})_{0\leq s \leq S}$ with $Y_{0}=\MM$ ,
$\sup_{0\leq s \leq S}L_{0}^{\B (s), \BB (s)} (X_{s}, Y_{s})$ is integrable.
\end{Prop}
\begin{proof}
Let $o\in M$ be a fixed reference point of $M$.
By Proposition \ref{Est0}(iii) and 
Proposition \ref{Est}(i), we have estimates
\begin{align*}
\mathrm{const.}
&\leq
L_{0}^{\B (s), \BB (s)} (X_{s}, Y_{s}) \\
&\hspace{0mm}\leq
\mathrm{const.} \Big\{
	\rho_{g(\B (s))} (X_{s},o)^{2} + \rho_{g(\BB (s))} (o, Y_{s})^{2}
\Big\}
+
\mathrm{const.}
\end{align*}
where the constants depend on $\sup_{0 \le t \le T} | \Ric |_{g(t)}$,
$T$, $\T_{1}$ and $\TT_{1}$ but not on $s$. 
Therefore, for our statement, it is sufficient to prove the integrability
of $\displaystyle \sup_{0\leq s\leq T} \rho_{g(\B (s))} (X_{s},o)^{2}$. 
By \cite[Theorem 2]{KP}, we see that
\begin{align*}
\D \rho(s,X_{s})
=
\Big\{
\frac{1}{2} \Delta_{g(\B (s))} - \frac{\partial}{\partial s}
\Big\} \rho (s,X_{s}) \D s
+
( U_{s}.e_{i} ) \rho (s,X_{s}) \D W_{s}^{i}
-
\D L_{s},
\end{align*}
denoting $\rho (s,x) :=  \rho_{g(\B (s))} (o, x)$, where $L_{s}$ is a nondecreasing continuous
process which increases only when $X_{s}$ belongs to the cut locus of $o$ with respect to $g(\B (s))$.
Note that the set of $s \in [0 , S]$ where $X_s$ is in $g(\B(s))$-cut locus 
has null Lebesgue measure and hence other quantities are also well-defined. 
Therefore we have
\begin{align*}
&
\D \rho(s,X_{s})^{2} \\
&=
\Big\{
\rho (s,X_{s})
\Big(
\Delta_{g(\B (s))} - 2\frac{\partial}{\partial s}
\Big) \rho (s,X_{s}) \D s
+  \sum_{i=1}^{d} | ( U_{s}.e_{i} ) \rho (s,X_{s}) |^{2}
\Big\} \D s \\
&\hspace{10mm}
- 2 \rho (s,X_{s}) \D L_{s}
+
2 \rho (s,X_{s}) \D \beta_{s}
\end{align*}
where $\D \beta_{s} = ( U_{s}.e_{i} ) \rho (s,X_{s}) \D W_{s}^{i}$ is a one-dimensional
Brownian motion.
Furthermore, Proposition 2 in \cite{KP} shows that
\begin{align*}
&
\rho (s,X_{s})
\Big\{
\frac{1}{2} \Delta_{g(\B (s))} - \frac{\partial}{\partial s}
\Big\} \rho (s,X_{s}) \\
&\hspace{5mm}\leq
\frac{d-1}{2} \rho (s,X_{s})
\Big\{
k_{1} \coth( k_{1} \cdot \rho (s,X_{s}) \wedge r_{1} ) + k_{1}^{2} \cdot \rho (s,X_{s}) \wedge r_{1}
\Big\}
\end{align*}
for some positive constants $k_{1}$ and $r_{1}$ from which we find that
there is some constant $c>0$ such that
\begin{align*}
\rho (s,X_{s})
\Big\{
\frac{1}{2} \Delta_{g(\B (s))} + \frac{\partial}{\partial s}
\Big\} \rho (s,X_{s})
\leq
c \rho (s,X_{s})^{2}
\quad
\text{for each $s\geq 0$ a.s.}
\end{align*}
We also find that
$\displaystyle
\sum_{i=1}^{d} | ( U_{s}.e_{i} ) \rho (s,X_{s}) |^{2} = 1.
$
So we shall apply a comparison argument between $\rho$'s SDE and
\begin{align*}
\left\{\begin{array}{l}
\D Z_{s} = \sqrt{4 Z_{s} \vee 0} \ \D \beta_{s} + ( c Z_{s} + 1 ) \D s , \\
Z_{0} = \rho (0,X_{0}) = \rho_{g( \T_{1} )} ( \M ,o ).
\end{array}\right.
\end{align*}
It is well-known that there is a global unique strong solution $Z=(Z_{s})_{s\geq 0}$ to the
above SDE and this satisfies $Z_{s} \geq 0$ for all $s\geq 0$ a.s. 
By the comparison theorem (e.g., see Theorem 1.1, Chapter V\!I in \cite{IW}), 
we see that $\rho (s,X_{s}) \leq Z_{s}$ for each $s\geq 0$ a.s. 
Since $\displaystyle \sup_{0 \leq u \leq s} Z_{u}^{p}$ is integrable for
any $p\geq 1$ (e.g., see Lemma 2.1, Chapter V in \cite{IW}), 
the conclusion follows. 
\end{proof}


\end{document}